\documentclass{amsart}
\usepackage{amsmath}
\usepackage{amsthm}
\usepackage{thmtools}
\usepackage{thm-restate}
\usepackage{graphicx} 
\usepackage[dvipsnames]{xcolor}
\usepackage{amssymb}
\usepackage{stackrel}
\usepackage[style=alphabetic,maxnames=99,minalphanames=3,maxalphanames=4]{biblatex}
\usepackage{mathrsfs}
\usepackage{outlines}
\usepackage{tikz-cd}
\usepackage{todonotes}
\usepackage{lacromay}
\usepackage{hyperref}
\usepackage{xypic}
\usepackage{enumitem}
\AtBeginBibliography{\footnotesize}

\let\oldtocsection=\tocsection

\let\oldtocsubsection=\tocsubsection

\let\oldtocsubsubsection=\tocsubsubsection

\renewcommand{\tocsection}[2]{\hspace{0em}\oldtocsection{#1}{#2}}
\renewcommand{\tocsubsection}[2]{\hspace{1em}\oldtocsubsection{#1}{#2}}
\renewcommand{\tocsubsubsection}[2]{\hspace{2em}\oldtocsubsubsection{#1}{#2}}

\title{The Stable Adjunction in $\bA^1$-Homotopy Theory}
\author{Ajay Srinivasan}
\address{\scriptsize{Department of Mathematics, University of Michigan, Ann Arbor, MI 48109}}
\email{avsriniv@umich.edu}
\date{\today}

\def\makeautorefname#1#2{\expandafter\def\csname#1autorefname\endcsname{#2}}
\def\makeautorefname#1#2{\expandafter\def\csname#1autorefname\endcsname{#2}}

\def\equationautorefname~#1\null{(#1)\null}
\makeautorefname{footnote}{footnote}%
\makeautorefname{item}{item}%
\makeautorefname{figure}{Figure}%
\makeautorefname{table}{Table}%
\makeautorefname{part}{Part}%
\makeautorefname{appendix}{Appendix}%
\makeautorefname{chapter}{Chapter}%
\makeautorefname{section}{Section}%
\makeautorefname{subsection}{Section}%
\makeautorefname{subsubsection}{Section}%
\makeautorefname{thm}{Theorem}%
\makeautorefname{sta}{Statement}%
\makeautorefname{cor}{Corollary}%
\makeautorefname{lem}{Lemma}%
\makeautorefname{prop}{Proposition}%
\makeautorefname{pro}{Property}
\makeautorefname{conj}{Conjecture}%
\makeautorefname{defn}{Definition}%
\makeautorefname{notn}{Notation}
\makeautorefname{notns}{Notations}
\makeautorefname{rem}{Remark}%
\makeautorefname{rems}{Remarks}%
\makeautorefname{quest}{Question}%
\makeautorefname{exmp}{Example}%
\makeautorefname{ax}{Axiom}%
\makeautorefname{claim}{Claim}%
\makeautorefname{assn}{Assumption}%
\makeautorefname{asses}{Assumptions}%
\makeautorefname{countcom}{Assumption}%
\makeautorefname{countcoms}{Assumptions}%
\makeautorefname{countmcom}{Assumption}%
\makeautorefname{con}{Construction}%
\makeautorefname{prob}{Problem}%
\makeautorefname{warn}{Warning}%
\makeautorefname{obs}{Observation}%
\makeautorefname{conv}{Convention}%

\newtheorem{thm}{Theorem}[section]

\newtheorem{prop}[thm]{Proposition}
\newtheorem{lem}[thm]{Lemma}

\theoremstyle{definition}
\newtheorem{defn}[thm]{Definition}

\newtheorem{exmp}[thm]{Example}

\newtheorem{rem}[thm]{Remark}

\newtheorem{obs}[thm]{Observation}

\newtheorem*{asses:monad}{Assumptions \ref{asses:monad}}

\newlist{salist}{enumerate}{1}
\setlist[salist]{
    label=\textbf{SA\arabic*},
    ref=SA\arabic*
}

\newcommand{\bm}{\mathbf{m}}
\newcommand{\bn}{\mathbf{n}}

\begin{document}

\maketitle

\begin{abstract}
    We prove a homotopical monadicity theorem for the adjunction between the suspension spectrum and zeroth space functors in motivic stable homotopy theory. 
    Our proof verifies that motivic stable homotopy theory satisfies the hypotheses of the general monadicity theorem of \cite{KMZ25}.
    In the process, we demonstrate six preliminary results in simplicial motivic homotopy theory, the main technical ingredient being a weak commutativity theorem between the zeroth space functor and realization of simplicial motivic spectra.
    We also elaborate on a general framework relating monadic algebras under op-lax maps of monads. This monadic framework is used in the proof of the main results and may be of independent interest as well. 
    These simplicial results and the accompanying monadic framework may provide tools toward a conjectured operadic recognition principle for motivic infinite loop spaces.  
\end{abstract}

\tableofcontents

\section*{Introduction}

At the crux of stable homotopy theory is the adjunction between the suspension spectrum and zeroth space functors. The same is true in motivic stable homotopy theory. This article analyzes the monadic structure of that stable adjunction in motivic homotopy theory. Consider a left adjoint functor $F \colon \aC \rtarr \aD$ with right adjoint $G \colon \aD \rtarr \aC$. The functor $\GA \colon = GF$ on $\aC$ has a canonical monad structure, where the product is induced by the counit $\epz$ of the adjunction.
$\GA$ is called the \emph{adjunction monad}. Let $\GA[\aC]$ denote the category of $\GA$-algebras in $\aC$. For every $Y \in \aD$, $GY$ has a canonical $\GA$-action given by $G\epz$, and as such $G$ induces a functor $G_{\GA} \colon \aD \rtarr \GA[\aC]$.

The \emph{monadicity} of the adjunction $(F,G)$, in the general sense that we use the word here, will refer to the extent to which the functor $G_{\GA} \colon \aD \rtarr \GA[\aC]$ is an equivalence of categories.
The result in \cite{Beck25} provides necessary and sufficient conditions for $G_{\GA}$ to be an equivalence on the nose. This result is often called \emph{classical} monadicity.
However in many settings, including ours, the conditions for classical monadicity fail and a specificized homotopical approach is more enlightening as to the monadicity of an adjunction. The recent framework of \cite{KMZ25} provides one roadmap to such a form of \emph{homotopical} monadicity. Roughly speaking, this paper here establishes that the stable adjunction in motivic homotopy theory is monadic \emph{up to linked $\bA^1$-homotopy}. The analogous result in topology has roots dating back to \cite{May72,May09Einfty}, but is dealt more carefully in the general framework in \cite{KMZ25}. The approach here is to collect enough results to plug motivic stable homotopy theory 
into the machine of \cite{KMZ25}. Through our development, we resolve what may be an important step towards a conjectural operadic recognition principle for motivic infinite loop spaces \cite{vopen}.

Let $\sT$ denote the Morel-Voevodsky model category of algebraic spaces, this will be our category of \emph{spaces}. In analogy with sequential $\OM$-spectra in topology, by \emph{spectra} we could pick Voevodsky's coordinatized $\bT$-spectra. However, for convenience with respect to the loop and suspension adjunction, we instead must work in a coordinate-free setting developed by \cite{Hu03}.
We will call Hu's model category of coordinate-free spectra $\sS$, and refer to its objects as \emph{spectra}. The functor $\OM^{\infty} \colon \sS \rtarr \sT$, obtained by taking the ``zeroth'' space of a spectrum, is right adjoint. The \emph{infinite loop spaces} will be those spaces that are in the image of $\OM^{\infty}$. The left adjoint of $\OM^{\infty}$, denoted $\SI^{\infty} \colon \sT \rtarr \sS$, is the motivic suspension spectrum in our setting.
The adjunction $(\SI^{\infty},\OM^{\infty})$ is what we refer to as the \emph{stable adjunction} in motivic homotopy theory, and it is the central object of study in this article.

Let $\GA = \OM^{\infty} \SI^{\infty}$ denote the monad associated to the adjunction $(\SI^{\infty},\OM^{\infty})$. Our first theorem utilizes the framework of \cite{KMZ25} to produce a first step towards homotopical monadicity for the stable motivic adjunction. Analogous to the topological story, the main results here pertain to the subcategory $\sS_c$ of \emph{connective} spectra (also called \emph{very effective} spectra elsewhere), which forms the homologically positive part of a $t$-structure on $\sS$.
\begin{restatable}{thm}{maintheorem}\label{thm:0.1}
    There is a functor $\bE \colon \GA[\sT] \rtarr \sS_c$ and a functor $\mathrm{Bar} \colon \GA[\sT] \rtarr \sT$, denoted $Y \mapsto \ol{Y}$, with the following properties.
    \begin{enumerate}
        \item $\bE \colon \GA[\sT] \rtarr \sS_c$ and $\OM^{\infty}_{\GA} \colon \sS_c \rtarr \GA[\sT]$ preserve weak equivalences. 
        \item For $E \in \sS_c$ and $Y \in \GA[\sT]$, we have a natural homotopy equivalence $\ze$ and weak equivalences $\ga$ and $\xi$ in the following directions:
    \[\xymatrix@1{Y & \ar[l]_-{\ze} \overline{Y} \ar[r]^-{\ga} & \OM^{\infty}_{\GA}\bE Y}\]
    \[\xymatrix@1{ \bE \OM^{\infty}_{\GA} E \ar[r]^-{\xi} & E.}\]
    \end{enumerate}
\end{restatable}

Theorem \ref{thm:0.1} however, does not establish an equivalence between the homotopy categories of $\GA$-algebras and connective spectra, and as such cannot be branded as homotopical monadicity. The functor $\mathrm{Bar}$ in the theorem alludes to the two-sided monadic bar resolution of $\GA$-algebras, built as the realization of a simplicial $\GA$-algebra---and as such we would have liked a functor $\mathrm{Bar} \colon \GA [\sT] \rtarr \GA[\sT]$. Similarly, we may have also liked for the weak equivalences in the theorem to 
be maps of $\GA$-algebras. In such a setting, we could immediately descend to homotopy categories to conclude that $(\bE,\OM^{\infty}_{\GA})$ induces an equivalence between homotopy categories. Unfortunately, due to a profound noncommutativity between $\OM^{\infty}$ and the realization of simplicial objects, there is no canonical $\GA$-action on $\ol{Y}$ for a $\GA$-algebra $Y$. So we do not quite have monadicity up to homotopy.

We resolve this gap by developing an abstract linking theory for monads connected by op-lax maps. Instantiating this abstract theory to our setting, we are able to concretely compare $\GA$-algebras and simplicial $\GA$-algebras. As a consequence, we show that although $\ol{Y}$ is not a $\GA$-algebra, it is the realization of a simplicial $\GA$-algebra whose algebra structure is \emph{linked} via realization to that of the $\GA$-algebras $\OM^{\infty}_{\GA} \bE Y$ and $Y$. We also introduce the idea of \emph{linked} maps between algebras. These are weakenings of algebra maps that still retain some monadic structure,
by virtue of being maps \emph{over} a linked algebra. Our second theorem is a strengthening of Theoerem \ref{thm:0.1} with this linked language.

\begin{restatable}{thm}{maintheoreM}\label{thm:0.2}
Let $(Y,\theta)$ be a $\GA$-algebra. There are maps
\[
\xymatrix@1{
    Y &  \overline{Y} \ar[r]^-{\ga} \ar[l]_-{\ze} & \OM^{\infty}_{\GA} \bE Y 
}
\]
\noindent where the following hold.
\begin{enumerate}
    \item Forgetting monadic structures, $\ze$ is a homotopy equivalence and $\ga$ is a weak equivalence.
    \item $\ze$ and $\ga$ are maps that link $\overline{Y}$ to the $\GA$-algebras $Y$ and $\OM^{\infty}_{\GA}\bE Y$ respectively.
    \item For $\nu$ the homotopy inverse of $\ze$, the composite $\ga \circ \nu$ is a linked weak equivalence between $Y$ and $\OM^{\infty}_{\GA} \bE Y$.
\end{enumerate}
\end{restatable}

Theorem \ref{thm:0.2} allows us to make the following concrete statement about homotopy categories, a precisification of being ``monadic up to linked $\bA^1$-homotopy.''

\begin{restatable}{cor}{maincor}\label{cor:0.3}
    The full subcategory of $\GA$-algebras in $\sT$ localized at the linked weak equivalences is equivalent to the homotopy category of connective spectra.
\end{restatable}

The paper is structured as follows. In Section \ref{sec:prelim} we discuss preliminaries. Sections \ref{subsec:mvf}, \ref{sec:coordfree}, and \ref{subsec:local} offer relevant background on Morel-Voevodsky algebraic spaces and on Po Hu's coordinate-free spectra. Section \ref{subsec:adjunction} defines the stable adjunction and elaborates on some of its formal and homotopical properties, particularly with respect to the category $\sS_c$ of connective spectra.
The proofs of the main theorems, owing to \cite{KMZ25}, rely on methods from simplicial homotopy theory. In Section \ref{subsec:salist}, we outline five properties of simplicial objects (\textbf{\ref{salist:SA1}}-\textbf{\ref{salist:SA5}}) that are essential to the proofs. The verification of these properties constitute all of Section \ref{sec:simpob}.
Properties \textbf{\ref{salist:SA1}}-\textbf{\ref{salist:SA5}} are easier to prove of the lot, and are verified in Section \ref{subsec:simpobez}. The hardest and perhaps conceptually most insightful property in the list is \textbf{\ref{salist:SA6}}, which describes the weak commutativity of realization and $\OM^{\infty}$. This is precisely the source for the weak commutativity of $\GA$ and $|-|$. The analogous result in topology plays a crucial role in the operadic recognition principle for infinite loop spaces \cite{May72}, and as such we believe
resolves an important step towards an operadic recognition principle for motivic infinite loop spaces. Pursuing this line of inquiry towards operadic monads is the subject of a sequel paper. \textbf{\ref{salist:SA6}} is proved in Section \ref{subsec:simpobhard} and Section \ref{subsec:rf}, where Rezk's reformulation of the compatibility of realizations with homotopy pullbacks \cite{Rezk14} plays a central part.
With the preliminaries in place, Section \ref{sec:thm} develops the required bar resolutions of $\GA$-algebras, and an abstract linking theory to complete the proofs of Theorems \ref{thm:0.1} and \ref{thm:0.2}. 



\section*{Acknowledgements}
Much of this article was developed while I was a visitor at the University of Chicago in the summers of 2024 and 2025. 
I am indebted to Peter May for his support through the summer program at Chicago, and for his efforts in this project, which was very much joint work with him.
My heartfelt gratitude also goes out to Alicia Lima, who co-advised me through the summer of 2024. I thank Paul Goerss for clarifications on simplicial path-space fibrations and Aravind Asok for pointing me to references on the commutativity of realization and homotopy pullback, and for continued discussion over the years on the social history of mathematics. 
The author was partially supported by NSF grant DGE-2241144 in the preparation of this draft.

\section{Preliminaries}\label{sec:prelim}
\subsection{The Morel-Voevodsky framework}\label{subsec:mvf}
This section recalls relevant details of the model category of algebraic spaces constructed by Morel and Voevodsky in \cite{MV99}. Let $k$ be a perfect field and $S$ a smooth affine integral scheme of finite type over $k$ (in particular, we ask that $S = \Spec(R)$ for a Noetherian integral domain $R$). We define $\mathrm{Sm}_S$ to be the category of smooth schemes of finite type over $S$. Whenever we refer to a scheme, we mean an object of $\mathrm{Sm}_S$. The category $\mathrm{Sm}_S$ is infeasible for homotopy theory since it is not cocomplete, so one cannot use the objects of $\mathrm{Sm}_S$ to mean ``spaces.'' The fix for this is to consider simplicial sheaves over $\mathrm{Sm}_S$ as spaces. To speak of sheaves, one endows $\mathrm{Sm}_S$ with an appropriate topology. \cite{MV99} identified the Nisnevich topology, finer than the Zariski but coarser than the \'Etale topology on $\mathrm{Sm}_S$, to be a convenient choice. It is defined so that the covering sieves are those families of \'etale morphisms $\phz_i \colon \{U_i\} \rtarr X$ such that for any $x \in X$, there exists $i$ and $u \in U_i$ with the corresponding map on residue fields being an isomorphism (that maps to $x$ with the same residue field). 
We write $(\mathrm{Sm}_S)_{Nis}$ to denote the site of $\mathrm{Sm}_S$ equipped with the Nisnevich topology. The definition of the Morel-Voevodsky category of spaces over $S$ is then\footnote{Notation: for a Grothendieck site $(\aC,\tau)$ we write $\mathrm{sSh}(\aC_{\tau})$ to mean the category of simplicial $\mathbf{Set}$-valued sheaves on $\aC$. When there is no confusion, we will call the objects of $\mathrm{sSh}(\aC_{\tau})$ simplicial sheaves on $\aC$.}: \[\mathrm{Spc}(S) := \sSh((\mathrm{Sm}_S)_{Nis}).\]

We can also define the category of based spaces over $S$ to be $\mathrm{Spc}(S)_{\bu} = S \downarrow \mathrm{Spc}(S)$. We write $\sT$ to mean this category and we call its objects spaces. The presheaf represented by an object $X \in \mathrm{Sm}_S$ is in fact a sheaf in the Nisnevich topology (this is due to \cite[VII.2a]{SGA4V2}). We also use $\sT_{k}$ to mean the category $\mathrm{Spc}(k)_{\bu}$ of spaces over $k$.\\


Now we look at some categorical properties of $\sT$ before moving on to its homotopical properties and its model structure. First, $\sT$ is complete and cocomplete. Second, $\sT$ is tensored and cotensored over itself. That is to say, there is a smash product $\wedge \colon \sT \times \sT \rtarr \sT$, and there are internal Hom-space functors $\underline{\Hom}_{\sT} (X, -)$ (right adjoint to $- \wedge X \colon \sT \rtarr \sT$) defined in the usual way. Finally, by the co-Yoneda lemma we know that every presheaf over $\mathrm{Sm}_S$ is a small colimit of schemes. Passing to sheafification, which is left adjoint, we can conclude that every Nisnevich sheaf over $\mathrm{Sm}_S$ is a small colimit of schemes. Up to homotopy, we can make a similar statement about how to build simplicial sheaves. The following compactness property \cite[Lemma 17.2]{Hu03} is often used in conjunction with co-Yoneda type formulae to reduce some Hom-computations down to the scheme level (for example, with the case of spectrification in Section \ref{sec:coordfree}).  
\begin{prop}\label{prop:1.2}
Every object of $\mathrm{Sm}_S$ is compact in $\sT$
\end{prop}

We now recall aspects of the model structure on $\sT$ and the notion of an $\bA^1$-homotopy defined in \cite{MV99}. We write $\bA^1$ to mean $\bA^1_S$, the affine line over $S$. Two maps $f, g \colon X \rtarr Y$ in $\sT$ are said to be $\bA^1$-homotopic if there is a map $h \colon \bA^1 \times X \rtarr Y$ in $\mathrm{Spc}(k)$, such that the composite $X \xrightarrow{\io_0} \bA^1 \times X \xrightarrow{h} Y$ is $f$, and $X \xrightarrow{\io_1} \bA^1 \times X \rtarr Y$ is $g$ (where $\io_j$ is induced by the inclusion of $S$ into $\bA^1_S$ as $j$). 

\begin{defn}
A map $f \colon X \rtarr Y$ in $\sT$ is said be a:
\begin{enumerate}
\item simplicial cofibration, if it is a monomorphism.
\item simplicial weak equivalence if for every $x^* \in \mathrm{Sm}_S$, the induced map\\ $f^* \colon x^*(X) \rtarr x^*(Y)$ is a weak equivalence of simplicial sets.\footnote{Here we think of an object $x$ of a site $\oS$ as a functor $x^* \colon \mathrm{Sh}(\oS) \rtarr \mathbf{Set}$.}
\item simplicial fibration if it has the right lifting property (RLP) with respect to all trivial simplicial cofibrations.
\end{enumerate}
\end{defn}

The above definitions provide a model structure on $\sT$ known as the simplicial model structure \cite[Corollary 2.7]{Jardine87simppr}. Denote by $\oH_s(\sT)$ the homotopy category of $\sT$ with respect to the simplicial model structure. The appropriate model structure that $\bA^1$-homotopy theory is set in comes from localizing the simplicial model structure with respect to $\bA^1$-homotopy. The first step towards such a localization is the notion of $\bA^1$-local spaces.

\begin{defn}
We say $X \in \sT$ is $\bA^1$-local if for every $Y \in \sT$, the map on hom-classes in the simplicial homotopy category $[Y,X]_{\oH_s(\sT)} \rtarr [Y \times \bA^1, X]_{\oH_s(\sT)}$ induced by the projection $Y \times \bA^1 \rtarr Y$ is a bijection.
\end{defn}

\begin{defn}
A map $f \colon X \rtarr Y$ in $\sT$ is said to be an
\begin{enumerate}
\item $\bA^1$-cofibration if it is a monomorphism.
\item $\bA^1$-weak equivalence if for any $\bA^1$-local $Z \in \sT$, the map on homotopy classes $f^* \colon [Y,Z]_{\oH_s (\sT)} \rtarr [X,Z]_{\oH_s(\sT)}$ is a bijection.
\item $\bA^1$-fibration if it has the right lifting property with respect to all trivial cofibrations.
\end{enumerate}
\end{defn}

The $\bA^1$-local definitions above prescribe a model structure on $\sT$ that admits functorial factorization \cites[Theorem 2.21]{MV99}[Section 2]{Jardine87simppr}. 
We denote by $\oH (\sT)$ the homotopy category of $\sT$ associated to this model structure.
As promised earlier, we have the following characterization of a space up to simplicial weak equivalence as a colimit \cite[Lemma 17.3]{Hu03}. 
\begin{prop}\label{prop:1.5}
Every object in $\sT$ is simplicially weak equivalent to a homotopy colimit of small colimits of smooth schemes. 
\end{prop}

We close by discussing the $\bA^1$-homotopy sheaves, $\pi^{\bA^1}_n(-,*)$, of a space. Let $X \in \sT$. For $n=0$, define $\pi_0^{\bA^1}(X)$ to be the sheaf associated to the presheaf $U \mapsto [U,X]_{\oH(\sT)}$ (for $U \in \mathrm{Sm}_S$). For $n \geq 1$, and the given basepoint $* \colon S \monoto X$, define $\pi_n^{\bA^1}(X,*)$ to be the sheaf associated to the presheaf $U \mapsto [\SI^n U_+, X]_{\oH(\sT)}$. Here $\SI^n$ denotes the simplicial suspension. We have the following motivic analog of Whitehead's theorem \cite[Proposition 2.14]{MV99}.

\begin{thm}\label{thm:1.1}
A map of spaces $f \colon X \rtarr Y$ in $\sT$ is an $\bA^1$-weak equivalence iff it induces isomorphisms of sheaves $\pi_0^{\bA^1}(X) \xrightarrow{\sim} \pi_0^{\bA^1}(Y)$, and $\pi_n^{\bA^1}(X,*) \xrightarrow{\sim} \pi_n^{\bA^1}(Y,\star)$ for every choice of basepoints $*$ for $X$ and $\star$ for $Y$ (as objects in $\mathrm{Spc}(S)$). 
\end{thm}

A most helpful computational property of the $\bA^1$-motivic homotopy sheaves is that of the long exact sequence associated to a homotopy fiber sequence analogous to the one in topology.

\begin{thm}\label{thm:1.2}
If $F \rtarr E \rtarr B$ is a homotopy fiber sequence of spaces, then there is a natural long exact sequence of homotopy sheaves:
\[\cdots \rtarr \pi^{\bA^1}_{n+1}(B) \rtarr \pi^{\bA^1}_n(F) \rtarr \pi_n^{\bA^1}(E) \rtarr \pi^{\bA^1}_n(B) \rtarr \cdots\]
\end{thm}

\noindent This is a natural long exact sequence first obtained at the presheaf level, passed through Nisnevich sheafification (which is exact). In the particular case of an inclusion of a subspace $\iota \colon A \monoto Y$, and the homotopy fiber of this inclusion, we get an analog of the long exact sequence of relative homotopy groups.\\

Fixing terminology for the rest of the paper, we say that $X \in \sT$ is \emph{connected} if $\pi_0^{\bA^1}(X)$ is terminal in the category of sheaves of sets (i.e. a constant sheaf at a singleton set). Similarly we say $X \in \sT$ is $n$-\emph{connected} if $\pi_q^{\bA^1}(X)$ is terminal for $q \leq n$ (in the category of sheaves of sets for $q=0$, of groups for $q=1$, and of abelian groups for $q \geq 2$). We say that a map $f \colon X \rtarr Y$ in $\sT$ is $n$-\emph{connected} if its homotopy fiber $Ff$ is $(n-1)$-connected. By Theorem \ref{thm:1.2} this means that $f$ induces isomorphisms $\pi_q^{\bA^1}(X) \xrightarrow{\sim} \pi_q^{\bA^1}(Y)$ for $q < n$ and a surjection $\pi_n^{\bA^1}(X) \epito \pi_n^{\bA^1}(Y)$. We say that a pair $(X,A)$ (i.e. an inclusion $\iota \colon A \rtarr X$ in $\sT$) is $n$-\emph{connected} if $\iota$ is $n$-connected.\\

\subsection{Po Hu's category of coordinate-free motivic spectra}\label{sec:coordfree}
In this section, we recall the construction of the coordinate-free spectra of \cite{Hu03}. Recall that our base scheme was $S = \Spec(R)$ for a Noetherian integral domain $R$. A natural choice for our coordinate-free indexing universe would be an infinite-dimensional affine space over $S$. We define our universe to be the countably infinite-dimensional $R$-module $\oU = \bigoplus_{\infty} R$. \footnote{For our purposes, it is okay to fix a universe here. \cite[Chapter 8]{Hu03} presents the theory for changes of universe.} We say $Z$ is a \emph{finite dimensional subspace} of $\oU$ if it is a finitely generated projective submodule of $\oU$ such that the inclusion $Z \monoto \oU$ splits. Fix a basis $\{e_1, e_2, \dots\}$ for $\oU$ and define $T_n$ to be the free submodule generated by $\{e_1, \dots, e_n\}$ so that $\oU \iso \colim_n T_n$ as an $R$-module. Now to think of $\oU$ as a space over $S$, it suffices to note that every finitely generated projective $R$-module can be thought of as a vector bundle over $S$. Each $T_n$ then corresponds to the trivial bundle $\bA^n_S$, the $n$-dimensional affine space over $S$. To define $\oU$ as a space (over $S$), we only need to write $\oU \iso \colim_n T_n \iso \colim_n \bA^n_S \iso \bA^{\infty}_S$. Note that $\oU$ is not really an object of $\mathrm{Sm}_S$, but it is an ind-scheme over $S$. Similarly, a finitely generated projective submodule $Z$ of $\oU$ can be thought of as a finite-dimensional subbundle of $\oU$ over $S$.\\

Our spectra will be indexed over the cofinite subspaces of $\oU$. These are the subspaces of $\oU$ of finite codimension, i.e. the split projective submodules of finite codimensions thought of as vector bundles over $S$. To put this definition on rigorous footing, P. Hu first defines the Grassmannian $Gr_{S}^{cof}(\oU)$ of cofinite subspaces of $\oU$. We begin with the Grassmannian $Gr_{k}^{cof}(\oU_{k})$ over $k$ of cofinite subspaces of $\bA^{\infty}_{k}$, and write $Gr_{S}^{cof}(\oU) := Gr_{k}^{cof}(\oU_{k}) \times_{\Spec(k)} S$. For $N \in \bN$ and $m \leq N$, define the Grassmannian $Gr_{k}^{m}(\bA^N)$ to be the space (over $k$) of subspaces of $\bA^N_k$ whose direct sum with $\bA^m_k$ is $\bA^N_k$ (here $\bA^m_k \monoto \bA^N_k$ by the inclusion into the first $m$ coordinates of $\bA^N_k$). Identifying $\oU_{k}$ with $\colim_N \bA^N_k$, we then define $Gr_{k}^{cof}(\oU_{k})$ to be the colimit $\colim_m \lim_{N \geq m} Gr_{k}^m(\bA^N)$. The limit is taken over the restriction maps $Gr_{k}^{m}(\bA^{N+1}) \rtarr Gr_{k}^m (\bA^N)$ taking $V$ to $V \cap \bA^N_k$ (here we think of $\bA^N_k \subset \bA^{N+1}_k$ as inclusion into the first $N$ coordinates). We say that a map $V \colon S \rtarr Gr^{cof}_{k}(\oU_{k})$ over $k$ is a cofinite subspace of $\oU$, in other words the cofinite subspaces of $\oU$ precisely correspond to the split projective submodules of $\oU$ of finite codimension. We similarly define an $n$-dimensional subspace of $\bA^N_S$ to be a map $Z \colon S \rtarr Gr_{k}(n,\bA^N_{k})$ (here $Gr_{k}(n,\bA^N_{k})$ is the Grassmannian of $n$-dimensional subspaces over $k$ of $\bA^N_{k}$). Again, this is just to say that the finite-dimensional subspaces are the finite-dimensional split projective submodules. For $U$ an affine space over $S$, and subspaces $V, W$ of $U$ with trivial intersection, we denote by $V \oplus W$ the internal direct sum of $V$ and $W$.\\

We are now ready to define our indexing category. Define the category of cofinite subspaces of $\oU$, written $\oC(\oU)$, to be the category whose objects are cofinite subspaces of $\oU$ and morphisms $Z \colon U \rtarr V$ are given by finite-dimensional subspaces $Z \subset U$ such that $V \oplus Z  = U$. The composite of morphisms $Z \colon U \rtarr V$ and $T \colon V \rtarr W$ is the finite-dimensional subspace $Z \oplus T$ of $U$ (note that $W \oplus (Z \oplus T) = U$). \cite[Lemma 2.3]{Hu03} shows that $\oC(\oU)$ is a small directed category.\\

Let $Z$ be a finite-dimensional space (alias finite-dimensional split projective module) over $S$. Thought of as a space, $Z$ looks like a vector bundle over $S$ equipped with a rational point $0 \colon S \rtarr Z$ (alias the zero-section of $Z$ as a vector bundle). We define the space $S^Z$ as the pushout \[\xymatrix{
Z \setminus \{0\} \ar[r] \ar[d] & Z \ar[d]\\
S \ar[r] & S^Z.
}\]
\noindent We think of $S^Z$ as the one-point compactification of $Z$ over $S$. For $X \in \sT$, we write $\OM^Z X$ to mean $\underline{\Hom}_{\sT} (S^Z, X)$. We will also denote the adjoint construction $S^Z \sma X$ by $\SI^{Z} X$.

\begin{defn}
A \emph{prespectrum} $D$ is a family of spaces $\{D_U\}_{U \in \oC(\oU)}$, together with morphisms $\rho^{U}_{V,Z} \colon D_U \rtarr \OM^ZD_V$ (for every morphism $Z \colon U \rtarr V$ in $\oC(\oU)$) satisfying the following conditions:
\begin{enumerate}[label=\alph*)]
\item $\rho^{U}_{U,0} = \mathrm{id}_{D_U}$
\item For every $Z \colon U \rtarr V$, and $T \colon V \rtarr W$ in $\oC(\oU)$, \[\rho^U_{W, Z \oplus T} = (\OM^Z \rho^{V}_{W,T}) \circ \rho^{U}_{V, Z}.\]
\end{enumerate}
\noindent The morphisms $\rho^{U}_{V,Z}$ are called the \emph{structure maps} of $D$. Given prespectra $D$ and $D'$, we call a collection of maps $\{f_U \colon D_U \rtarr D'_U\}_{U \in \oC(\oU)}$ whose elements are compatible with the structure maps, a map of prespectra from $D$ to $D'$, written $f \colon D \rtarr D'$. We will denote by $p\sS$ the category of prespectra.
\end{defn}

\begin{defn}\label{defn:spectra}
We say that a prespectrum $E$ is a \emph{spectrum} if its structure maps are all isomorphisms. A map of prespectra between two spectra is called a map of spectra. We denote by $\sS$ the category of spectra.
\end{defn}

All limits in $\sS$ (and all limits and colimits in $p\sS$) can be computed spacewise, just as in the topological case. The question of colimits, can be answered as follows. There is a forgetful functor $R \colon \sS \rtarr p\sS$, given by forgetting that the structure maps are isomorphisms. $R$ admits a left adjoint $L$ that we call spectrification. Thanks to the Nisnevich topology (namely the compactness of $S^Z$ as an object in $\sT$ \ref{prop:1.2}), the story here is much simpler than in topology where one must first pass through inclusion prespectra and Freyd's adjoint functor theorem.

\begin{prop}\label{prop:2.1}
There exists a functor $L \colon p \sS \rtarr \sS$ called spectrification, defined by $D \mapsto LD$ where $(LD)_U = \colim_{(V,Z) \in U \downarrow \oC(\oU)} \OM^Z D_V$ with prescribed structure maps $\rho^{U}_{W,T} : (LD)_U \xrightarrow{\sim} \OM^T (LD)_W$. Here the colimit is taken over the maps $\OM^Z D_V \rtarr \OM^{Z'}D_{V'}$ given by $\OM^{Z}\rho^{V}_{V', T}$ for maps $T \colon (V,Z) \rtarr (V',Z')$. 
\end{prop}

\begin{proof}
We only demonstrate why $\{(LD)_U\}_{U \in \oC(\oU)}$ is actually a spectrum. It suffices to prescribe structure maps $\rho^{U}_{W,T} : (LD)_U \xrightarrow{\sim} \OM^T (LD)_W$ for $T \colon U \rtarr W$ in $\oC(\oU)$. We begin with $(LD)_U$. 
\begin{equation*}
\begin{split}
(LD)_U 
&\iso \colim_{(V,Z) \in W \downarrow \oC(\oU)} \OM^{T \oplus Z} D_V\\
&\iso \OM^T \colim_{(V,Z) \in W \downarrow \oC(\oU)} \OM^Z D_V\\
&= \OM^T(LD)_W
\end{split}
\end{equation*}
Here we have used Proposition \ref{prop:1.2} to see that $S^T$ is compact in $\sT$, allowing us to commute $\OM^T$ and the colimit. 
\end{proof}

As a left adjoint, spectrification commutes with colimits. Therefore we may compute colimits in $\sS$ by first computing them at the prespectrum level and then spectrifying. The smash product $X \wedge E$ of a space $X$ with a spectrum $E$ is given by taking smash products spacewise and then spectrifying. We also have the function spectrum $F(X,E)$ defined by $F(X,E)_U = \underline{\Hom}_{\sT} (X,E_U)$ for each $U \in \oC(\oU)$. These constructions show that $\sS$ is tensored and cotensored over $\sT$. One might wonder how the category $\sS$ compares with Voevodsky's original construction of the category of sequential $\bT$-spectra in \cite{voev96pre}. These concerns are addressed in \cite[Proposition 3.4 and Theorem 4.5]{Hu03}. These categories are in fact on the nose equivalent.

\subsection{The $\bA^1$-local stable model structure}\label{subsec:local}
In this section, we recall P. Hu's definition of the stable $\bA^1$-local model structure on $\sS$. There is first, a stable \emph{simplicial} model structure on spectra, which itself comes from levelwise and stable simplicial model structures on the \emph{prespectrum} level. 

\begin{defn}
We say a map $f \colon D \rtarr E$ in $p\sS$ is a levelwise simplicial weak equivalence or fibration if for every $V \in \oC(\oU)$, the $V$-th space map $f_V \colon D_V \rtarr E_V$ is a simplicial weak equivalence or simplicial fibration of spaces (over $k$) respectively. We say $f$ is a levelwise simplicial cofibration if it satisfies the left lifting property (LLP) with respect to every levelwise simplicial acyclic fibrations. 
\end{defn}

\cite[Proposition 5.10 and Chapter 15]{Hu03} show using the small objects argument that the distinguished classes of maps in the above definitions form a model structure on $p\sS$. We call this model structure the \emph{levelwise simplicial model structure} on $p\sS$. 

\begin{defn}
Let $f \colon D \rtarr E$ be a map in $p\sS$. We say $f$ is a:
\begin{enumerate}
\item stable simplicial cofibration if it is a levelwise simplicial cofibration.
\item stable simplicial weak equivalence if it induces a bijection of colimits:
\[\colim_{(W,Z) \in V \downarrow \oC(\oU)} [\SI^Z X_+, D_W]_{\oH_s(\sT_{k})} \iso \colim_{(W,Z) \in V \downarrow \oC(\oU)} [\SI^Z X_+, E_W]_{\oH_s(\sT_{k})}\]
\noindent for every $V \in \oC(\oU)$ and every $X \in \mathrm{Sm}_S$.
\item stable simplicial fibration if it satisfies the RLP with respect to all stable simplicial acyclic cofibrations.
\end{enumerate}
\end{defn}

The Bousfield-Friedlander theorem of \cite[Theorem A.7]{BF1978} shows that the above definitions form a model structure on $p\sS$. We call this model structure the \emph{stable simplicial model structure} on $p\sS$. We can now define the stable simplicial model structure on $\sS$ by the following definitions.

\begin{defn}
Let $f \colon D \rtarr E$ be a map in $\sS$. We say $f$ is a stable simplicial cofibration, weak equivalence, or fibration if it is one in the stable simplicial model structure on $p\sS$.
\end{defn}

\noindent The small objects argument from \cite[Chapter 15]{Hu03} applied to the above definitions show that they prescribe a model structure on $\sS$. We call this model structure the stable simplicial model structure on $\sS$, and we denote the associated homotopy category of $\sS$ by $\oH_s(\sS)$.\\

To $\bA^1$-localize the stable simplicial model structure on $\sS$, we first recall P. Hu's notion of an $\bA^1$-homotopy in $\sS$. Consider two maps $f, g \colon E \rtarr G$ in $\sS$. We say $f$ and $g$ are $\bA^1$-homotopic if there is a map $h \colon E \wedge \bA^1_+ \rtarr G$ whose composition with the maps $\iota_0, \iota_1 \colon E = E \wedge S^0 \rtarr E \wedge \bA^1_+$ (one sending the non-base-point of $S^0$ to $0$, and the other sending that point to $1$ in $\bA^1$) can be identified with $f$ and $g$ respectively. We define $\bA^1$-local spectra analogous to $\bA^1$-local spaces.

\begin{defn}
A spectrum $G \in \sS$ is termed $\bA^1$-local if for every $E \in \sS$, the map $E \wedge \bA^1_+ \rtarr E$ induced by the projection map $\bA^1_+ \rtarr S^0$ induces a bijection of homotopy classes:
\[[E,G]_{\oH_s(\sS)} \rtarr [E \wedge \bA^1_+, G]_{\oH_s(\sS)}\]
\end{defn}

\begin{defn}
Let $f \colon E \rtarr E'$ be a map in $\sS$. We say that $f$ is an:
\begin{enumerate}
\item $\bA^1$-cofibration if it is a stable simplicial cofibration.
\item $\bA^1$-weak equivalence if for every $\bA^1$-local $G$, the map of homotopy classes $[E',G]_{\oH_s(\sS)} \rtarr [E,G]_{\oH_s(\sS)}$ is a bijection.
\item $\bA^1$-fibration if it satisfies the RLP with respect to every acyclic $\bA^1$-cofibration.
\end{enumerate}
\end{defn}

The above definitions form a model structure on $\sS$ that we call the \emph{stable $\bA^1$-local model structure} on $\sS$ \cite[Theorem 6.4]{Hu03}. The proof is mostly analogous to that of \cite[Theorem 2.2.21]{MV99}. 
We denote by $\oH(\sS)$ the homotopy category of $\sS$ associated to the stable $\bA^1$-local model structure.

\subsection{The $(\SI^{\infty},\OM^{\infty})$ adjunction}\label{subsec:adjunction}
In this section, we define the adjunction $(\SI^{\infty},\OM^{\infty})$ and study some of its properties. For any $V \in \oC(\oU)$, there is the $V$-th space functor $\OM^{\oU}_V \colon \sS \rtarr \sT$ given by $\OM^{\oU}_V E = E_V$. This functor has a left adjoint $\SI^{\oU}_V \colon \sT \rtarr \sS$, called the $V$-th shift desuspension of the suspension spectrum, whose construction we now recall. Let $X \in \sT$. For each $U \in \oC(\oU)$, define the space $D_V(X)_U = \bigvee_{Z \oplus U = V} \SI^Z X$, and $D_V(X)_U = S$ (a point) otherwise. If $W \subset U \subset V$ (so that $V \rtarr U \rtarr W$ in $\oC(\oU)$) such that $W \oplus T = U$, then for any $Z$ with $Z \oplus U = V$, we have $(T \oplus Z) \oplus W = V$. This furnishes compatible maps:
\[\SI^T D_V(X)_V = \bigvee_{Z \oplus U = V} \SI^{T \oplus Z} X \rtarr \bigvee_{Z' \oplus W = V} \SI^{Z'}X = D_V(X)_W.\]
By the $(\SI^Z, \OM^Z)$ adjunction, this furnishes a prespectrum $D_V(X) = \{D_V(X)_U\}_{U \in \oC(\oU)}.$ We define $\SI^{\oU}_V X := LD_V(X)$.\\

It remains to show that $\SI^{\oU}_V$ is left adjoint to $\OM^{\oU}_V$. It suffices to work on the level of prespectra. Let $f \colon D_V(X) \rtarr E$ be a map of prespectra (with $D_V(X)$ as defined above), i.e. a collection of compatible maps $f_U \colon D_V(X)_U \rtarr E_U$. The $V$-th map $f_V$ is then just a map $X \rtarr E_V$. Conversely, let $g \colon X \rtarr E_V$ be a map of spaces. Let $U$ be a cofinite subspace of $V$. For each map $Z \colon V \rtarr U$ (that is, a finite subspace $Z \subset V$ such that $Z \oplus U = V$), we obtain the map $\SI^Z X \xrightarrow{\SI^Z g} \SI^Z E_V \rtarr E_U$, where the last map is furnished by the structure maps of $E$. This prescription provides compatible maps $D_V(X)_U \rtarr E_U$, as needed.\\

We record the following key characterization of $\bA^1$-local spectra, and an analog of Whitehead's theorem for $\bA^1$-weak equivalences of spectra \cite[Lemma 6.12 and Proposition 6.13]{Hu03}.

\begin{thm}
A spectrum $G$ is $\bA^1$-local iff for every $X \in \mathrm{Spc}(S)$ that is a finite colimit of smooth schemes over $S$, and for every $V \in \oC(\oU)$, the map \[[\SI^{\oU}_V X_+, G]_{\oH_s(\sS)} \rtarr [\SI^{\oU}_{V}X_+ \wedge \bA^1_+, G]_{\oH_s(\sS)}\] \noindent induced by the projection $\bA^1_+ \rtarr S^0$ is a bijection.
\end{thm}

\begin{thm}\label{thm:2.2}
A map $f \colon G \rtarr G'$ in $\sS$ is an $\bA^1$-weak equivalence iff for every $X \in \mathrm{Spc}(S)$ that is a finite colimit of smooth schemes over $S$, and $V \in \oC(\oU)$, the induced map: $[\SI^{\oU}_V X_+, G]_{\oH(\sS)} \rtarr [\SI^{\oU}_V X_+, G']_{\oH(\sS)}$ is a bijection.
\end{thm}

\noindent For the rest of this article, we will find it customary to drop the ``$\bA^1$-'' when referring to the distinguished classes of maps in the model structures on $\sT$ and $\sS$.\\

We denote by $\SI^{\infty}$ the functor $\SI^{\oU}_{\oU} \colon \sT \rtarr \sS$, and by $\OM^{\infty}$ the functor $\OM^{\oU}_{\oU} \colon \sS \rtarr \sT$. We call $\SI^{\infty} X$ the \emph{suspension spectrum} of $X$, and $\OM^{\infty}E$ the \emph{$\oU$-th space} or the \emph{zeroth space} of $E$. We have just shown that the pair $(\SI^{\infty},\OM^{\infty})$ is an adjunction between $\sT$ and $\sS$. This is the categorical bedrock of this article. We iron out some formal preliminaries and notation now. Denote by $\GA$ the adjunction monad $\OM^{\infty}\SI^{\infty}$. The adjunction provides a unit map $\eta \colon \mathrm{id}_{\sT} \rtarr \GA$, and a counit map $\epz \colon \SI^{\infty}\OM^{\infty} \rtarr \mathrm{id}_{\sS}$. We write $\GA[\sT]$ for the category of $\GA$-algebras in $\sT$ (i.e. spaces in $\sT$ with prescribed $\GA$-actions). It is easy to note that $\OM^{\infty} E$ is a $\GA$-algebra for every $E \in \sS$. The $\GA$-action is provided by the composite $\GA \OM^{\infty} E = \OM^{\infty} \SI^{\infty} \OM^{\infty} E \xrightarrow{\OM^{\infty}\epz_E} \OM^{\infty}E$. On the other hand, $\SI^{\infty}$ is a $\GA$-functor in the sense that there is a natural map $\beta \colon \SI^{\infty} \GA \rtarr \SI^{\infty}$ defined by $\epz_{\SI^{\infty}X} \colon \SI^{\infty} \GA X \rtarr \SI^{\infty} X$ for every $X \in \sT$ (such that $\be \circ \SI^{\infty}\mu = \be \circ \be$ and $\be \circ \SI^{\infty}\et = \id$).

We denote by $\OM^{\infty}_{\GA}$ the functor $\sS \rtarr \GA[\sT]$ obtained by just applying $\OM^{\infty}$. $\OM^{\infty}_{\GA}$ is also a right adjoint, but its left adjoint is \emph{not} $\SI^{\infty}$. While maps of spectra $\SI^{\infty} X \rtarr E$ (for an algebra $X$) do correspond to maps of spaces $X \rtarr E_{\oU}$, they do not a priori correspond to \emph{maps of algebras}. The fix for this is to coequalize with respect to the $\GA$-action on $\SI^{\infty}$ (seen as a $\GA$-functor) and the $\GA$-action on an algebra $X$. For each $X \in \GA[\sT]$ with $\GA$-action $\theta$, define $\SI^{\infty}_{\GA}X$ to be the coequalizer 
\[\xymatrix{
\SI^{\infty}\GA X \ar@<-0.5ex>[r]_{\SI^{\infty}\theta} \ar@<0.5ex>[r]^{\beta_X} & \SI^{\infty}X \ar[r]& \SI^{\infty}_{\GA}X.
}\]
\noindent This construction defines a functor $\SI^{\infty}_{\GA} \colon \GA[\sT] \rtarr \sS$ left adjoint to $\OM^{\infty}_{\GA}$, and in doing so provides a coequalized adjunction $(\SI^{\infty}_{\GA},\OM^{\infty}_{\GA})$ between $\GA[\sT]$ and $\sS$.\\

We now define a subcategory of connective spectra that is especially well-behaved with respect to our adjunction. Connective spectra are considered in analogy with their namesakes from topology, where they are the spectra with trivial negative homotopy groups. Let $\sS_c$ denote the smallest full subcategory of $\sS$ containing all suspension spectra of smooth schemes over $S$ that is closed under homotopy colimits and extensions. Equivalently $\sS_c$ is generated by colimits and extensions of $\SI^{\infty}_+ X$ for $X \in \mathrm{Sm}_S$ (here $\SI^{\infty}_+ X$ is used as shorthand for $\SI^{\infty}X_+$). The objects of $\sS_c$ are referred to as very effective spectra in the literature (see \cite[Definition 5.5]{SO12veff}), but we call them \emph{$\OM^{\infty}$-connective spectra} or \emph{connective spectra} for short. $\sS_c$ is not a triangulated category itself, but it forms the homologically positive part of a $t$-structure on $\sS$. It is symmetric monoidal however under the smash product. By Proposition \ref{prop:1.5}, and since $\SI^{\infty}$ commutes with colimits as a left adjoint, $\SI^{\infty} \colon \sT \rtarr \sS$ takes values in $\sS_c$. We can now begin our homotopical investigation into the $(\SI^{\infty},\OM^{\infty})$ adjunction.

\begin{prop}
If $f \colon X \rtarr Y$ is a weak equivalence in $\sT$, then $\SI^{\infty}f \colon \SI^{\infty} X \rtarr \SI^{\infty} Y$ is a weak equivalence in $\sS$.
\end{prop}
\begin{proof}
Let $f \colon X \rtarr Y$ be a weak equivalence in $\sT$, so that for every $\bA^1$-local space $Z$, the induced map $[Y,Z]_{\oH_s(\sT)} \rtarr [X,Z]_{\oH_s(\sT)}$ is an isomorphism. Let $E$ be an arbitrary $\bA^1$-local spectrum in $\sS$. Note that $\OM^{\infty}E = E_{\oU}$ is $\bA^1$-local in $\sT$. This is because for every $W \in \sT$, $[\SI^{\infty}W,E] \rtarr [\SI^{\infty}W \wedge \bA^1, E]$ (induced by the projection) is an isomorphism. By the adjunction then, we can see that $[W,E_{\oU}] \rtarr [W \wedge \bA^1, E_{\oU}]$ (again induced by the projection) is an isomorphism. Now we know that $[Y,E_{\oU}]_{\oH_s(\sT)} \rtarr [X,E_{\oU}]_{\oH_s(\sT)}$ is an isomorphism. Passing to the adjunction, we have that $[\SI^{\infty}Y,E]_{\oH_s(\sS)} \rtarr [\SI^{\infty}X,E]_{\oH_s(\sS)}$ is an isomorphism as well. 
\end{proof}

The analogous result for $\OM^{\infty}$ follows from Theorem \ref{thm:2.2} and Theorem \ref{thm:1.1} (and Proposition \ref{prop:1.5}, which allows us to pass to finite colimits of smooth schemes and use \ref{thm:2.2}).
\begin{prop}\label{prop:2.3}
If $g \colon E \rtarr E'$ is a weak equivalence in $\sS$, then $g_{\oU} \colon E_{\oU} \rtarr E_{\oU}'$ is one in $\sT$.
\end{prop}

Moreover, a stronger result holds in the case of connective spectra---$\OM^{\infty}$ \emph{reflects} weak equivalences.
\begin{prop}\label{prop:2.13}
Let $g \colon E \rtarr E'$ be a map between connective spectra in $\sS$. The map $g$ is a weak equivalence in $\sS$ iff $g_{\oU} \colon E_{\oU} \rtarr E'_{\oU}$ is one in $\sT$. 
\end{prop}
\begin{proof}
It only remains to prove the backward direction, beginning with a map $g \colon E \rtarr E'$ such that $g_{\oU} \colon E_{\oU} \rtarr E'_{\oU}$ is a weak equivalence in $\sT$.
To reduce the problem, we invoke \cite[Proposition 6.10]{Hu03}--a descent principle for weak equivalences of spectra which states that (small) directed colimits of spectra preserve weak equivalences. 
In light of this reduction, we can restrict ourselves to the case of a weak equivalence $g_{\oU} \colon (\SI^{\infty} Y_+)_{\oU} \rtarr (\SI^{\infty} Y'_+)_{\oU}$, where $Y$ and $Y'$ are finite colimits of smooth schemes.

Let $X$ be a finite colimit of smooth schemes. By assumption, \[[\SI^{\oU}_{\oU}X_+, \SI^{\infty}Y_+]_{\oH(\sS)} \rtarr [\SI^{\oU}_{\oU}X_+, \SI^{\infty}Y'_+]_{\oH(\sS)}\] \noindent is a bijection. Let $V \in \oC(\oU) \setminus \{\oU\}$. 
We observe that for any finite colimit of smooth schemes $\oY$, $[\SI^{\oU}_V X_+, \SI^{\infty}\oY_+]_{\oH(\sS)}$ is trivial. Note the following chain of isomorphisms:
\begin{equation*}
    \begin{split}
        [\SI^{\oU}_V X_+, \SI^{\infty}\oY_+]_{\oH(\sS)} &\cong [X_+, (\SI^{\infty}\oY_+)_V]_{\oH(\sS)}\\
        &\cong \left[X_+, \colim_{(W,Z) \in V \downarrow \oC(\oU)}\OM^Z\bigvee_{T \oplus W = \oU}\SI^{T}\oY_+\right]_{\oH(\sS)},
    \end{split}
\end{equation*}
where $T \colon \oU \rtarr V$ in $\oC(\oU)$. Now using Proposition \ref{prop:1.2} on $X_+$ (or rather since finite colimits of compact objects are compact), we see that:
\begin{equation*}
    \begin{split}
        [\SI^{\oU}_V X_+, \SI^{\infty}\oY_+]_{\oH(\sS)} &\cong \colim_{(W,Z) \in V \downarrow \oC(\oU)} \left[X_+, \OM^Z \bigvee_{T \oplus W = \oU}\SI^{T}\oY_+\right]_{\oH(\sS)}\\
        &\cong \colim_{(W,Z) \in V \downarrow \oC(\oU)} \left[\SI^{Z} X_+, \bigvee_{T \oplus W = \oU}\SI^{T}\oY_+\right]_{\oH(\sS)}.
    \end{split}
\end{equation*}
\noindent To finish the argument, note that $\dim Z < \dim (T)=\mathrm{codim}(W)$ for every $T$ such that $T \oplus W = \oU$ (since $V \neq \oU$), and $\SI^{T} \oY_+$ is $(\dim(T)-1)$-connected\footnote{See remarks following \cite[Theorem 18]{morel12A1AT}.}. 
We conclude from Theorem \ref{thm:2.2} that $g$ is a weak equivalence in $\sS$.
\end{proof}

\subsection{Simplicial axioms}\label{subsec:salist}
The proof of the homotopical monadicity theorem, using the argument of \cite{KMZ25}, will rely on a derived homotopy theory defined by bar resolutions of $\GA$-algebras. 
The two-sided monadic bar construction used here is built simplicially, and as such we list properties of simplicial objects in $\sT$ and $\sS$ that will be essential for the proof.
The proofs of these properties constitute Section \ref{sec:simpob}.\\
\indent We will write $s\aC$ for the simplicial objects in a category $\aC$. For a functor $F \colon \aC \rtarr \aD$, we will write $F_*$ for the induced functor on 
simplicial objects $F_* \colon s\aC \rtarr s\aD$ given by applying $F$ objectwise on simplices.

\begin{salist}
    \item \label{salist:SA1} There exist realization functors $|-| \colon s\sT \rtarr \sT$ and $|-| \colon s\sS \rtarr \sS$ that are both left adjoints. The composite $|c_*|$ where $c_*$ is the constant simplicial object functor $\aC \rtarr s\aC$, is the identity for $\aC = \sT, \sS$.
    \item \label{salist:SA2} $|-|$ preserves homotopies on $s\sT$ and $s\sS$.
    \item \label{salist:SA3} $|-|$ preserves weak equivalences between Reedy cofibrant objects in both $s\sT$ and $s\sS$.
    \item \label{salist:SA4} For $K_* \in s\sT$, there is a natural isomorphism $\SI^{\infty}|K_*| \cong |\SI^{\infty}_* K_*|$.
    \item \label{salist:SA5} $|-| \colon s\sS \rtarr \sS$ takes levelwise connective spectra to connective spectra. 
    \item \label{salist:SA6} For $K_* \in s\sS$, the natural map $\ga \colon |\OM^{\infty}_* K_*| \rtarr \OM^{\infty} |K_*|$, given by the adjoint to the composite $\xymatrix@1{\SI^{\infty}|\OM^{\infty}_* E_*| \cong |\SI^{\infty}_*\OM^{\infty}_*E_*| \ar[r]^-{|\epz_*|} & |E_*|}$
    is, for levelwise connective $E_*$, a weak equivalence in $\sT$.
\end{salist}

\noindent We will also frequently use the following fact (which follows formally from \textbf{\ref{salist:SA1}}). 
\begin{obs}\label{obs:4.1}
For a simplicial object $X_* \in s\sT$ and an object $J \in \sT$, a map $f \colon X_0 \rtarr J$ such that $fd_0 = fd_1$
extends to a map $\xi_* \colon X_* \rtarr c_* J$ such that $\xi_0 = f$, and therefore induces a map $\xi:=|\xi_*| \colon |X_*| \rtarr J$ in $\sT$.
\end{obs}
It is a bit awkward to list properties that we have not yet proved as ``preliminaries,'' but the hope is that the reader will, until Section \ref{sec:simpob}, take their validity for granted. Without saying much about their proofs here, we will only say that while \textbf{\ref{salist:SA1}-\ref{salist:SA5}}
have relatively quicker proofs, \textbf{\ref{salist:SA6}} requires some labor. Consequently, we will sometimes refer to the former as the \emph{easy axioms}, and the latter as the \emph{hard axiom}.

\section{The proofs of Theorems \ref{thm:0.1} and \ref{thm:0.2}}\label{sec:thm}
\subsection{Proof of the crude monadicity theorem}
The proof of crude monadicity presented here is due to an upcoming article \cite{KMZ25} and discussions with J.P. May. First, we define the two-sided bar constructions that are essential for the proof. 
Let us return to the monad $(\GA, \mu, \et)$, where $\GA = \OM^{\infty}\SI^{\infty}$. We noted in Section \ref{sec:coordfree} that $\SI^{\infty}$ was a $\GA$-functor in the sense that there was a natural map $\be \colon \SI^{\infty} \GA \rtarr \SI^{\infty}$ defined by $\be_X = \epz_{\SI^{\infty}X}$.
We can also consider $\GA$ as a $\GA$-functor with the action map $\mu = \OM^{\infty}\epz \colon \GA \GA \rtarr \GA$. For any $\GA$-algebra $(Y,\theta)$, we have the two-sided bar constructions:
\[\overline{Y} = B(\GA,\GA,Y) \; \; \textrm{and} \; \; \bE Y = B(\SI^{\infty},\GA,Y),\]
where $\overline{Y} \in \sT$, and $\bE Y \in \sS$. Recall that $B(\GA,\GA,Y)$ is the realization of the simplicial object in $\sT$ (denoted $B_*(\GA,\GA,Y)$) whose $q$-simplices are given by $\GA^{q+1}Y$, and $B(\SI^{\infty},\GA,Y)$ is the realization of the simplicial object in $\sS$ (denoted $B_*(\SI^{\infty},\GA,Y)$ or $\bE_* Y$) 
whose $q$-simplices are given by $\SI^{\infty}\GA^qY$. The zeroth face is given by the action of $\GA$ (which is either $\mu$ or $\be$), the $i$-th face for $0 < i < q$ is given by $\mu$, and the $q$-th face is given by $\theta$. The degeneracies are all given by $\et \colon \id \rtarr \GA$. By Observation \ref{obs:4.1}, applied to $\tha \colon \GA Y \rtarr Y$, there is a natural simplicial map $\ze_* \colon B_{*}(\GA,\GA,Y) \rtarr c_* Y$ which 
induces a natural map $\ze \colon \ol{Y} \rtarr Y$. We will also sometimes find it notationally convenient to write $B_*(\GA,\GA,Y) = \OM^{\infty}_* \bE_* Y$ so that $\ol{Y} = |\OM^{\infty}_* \bE_* Y|$.
An extra degeneracy argument shows that:
\begin{lem}\label{lem:4.1}
For $Y \in \GA[\sT]$, the natural map $\ze \colon \overline{Y} \rtarr Y$ in $\sT$ induced by $\mu \colon \GA^{q+1} Y \rtarr Y$ is a homotopy equivalence with homotopy inverse $\nu \colon Y \rtarr \overline{Y}$ induced by the maps $\et \colon Y \rtarr \GA^{q+1}Y$.
\end{lem}
It is in this sense that we can see $\overline{Y}$ as a derived resolution of $Y \in \GA[\sT]$. However, there is a conceptual gap in the above result. $Y$ is a $\GA$-algebra, and the map $\mu \colon \GA^{q+1} Y \rtarr \GA^q Y$ is a map of $\GA$-algebras.
We would like to say that $\overline{Y}$ is a $\GA$-algebra as well, and $\ze$ is a map of $\GA$-algebras---so that we can really think of $\overline{Y}$ as a derived resolution of $Y$ in the category of $\GA$-algebras. However, $\OM^{\infty}$ does not commute with $|-|$ on the nose!
It is therefore not at all clear that there is a canonically induced $\GA$-algebra structure on $\overline{Y}$. We will soon elaborate on the kind of formal structure that $\overline{Y}$ possesses. 

Now we describe a conceptually dual procedure on $\sS$. Let $\hat{\GA} := \SI^{\infty} \OM^{\infty}$ denote the adjunction comonad and let $E \in \sS$. Observe that $\bE_*\OM^{\infty}_{\GA}E$ has $q$-simplices $\hat{\GA}^{q+1}E$ and all faces given by $\epz$. We define the following as a conceptual dual to $\overline{Y}$ for $Y \in \GA[\sT]$.
\[\hat{E} := B(\SI^{\infty},\GA,\OM^{\infty}_{\GA}E) = \bE \OM^{\infty}_{\GA} E.\]
\noindent Since $\SI^{\infty}Z$ is always connective for $Z \in \sT$, and $|-|$ takes levelwise connective spectra to connective spectra (by \textbf{\ref{salist:SA5}}), we see that $\bE Y$ is connective for any $Y \in \GA[\sT]$. 
Therefore $\hat{E}$ is connective for any $E \in \sS$. The following should be seen as the dual to Lemma \ref{lem:4.1}.
\begin{lem}\label{lem:4.3}
    For $E \in \sS$, the face maps $\epz \colon \hat{\GA}^{q+1} \rtarr \hat{\GA}^q$ induce a natural map $\xi \colon \hat{E} \rtarr E$ in $\sS$. If $E$ is connective, then $\xi$ is a weak equivalence.
\end{lem}
\begin{proof}
    The first part of the statement is a direct application of Observation \ref{obs:4.1} to $\epz \colon \hat{\GA}E \rtarr E$. Now suppose $E$ is connective. Since $\hat{E}$ is connective, by Proposition \ref{prop:2.13}, it suffices to show that $\OM^{\infty}\xi \colon \OM^{\infty}\hat{E} \rtarr \OM^{\infty} E$ is a weak equivalence. Consider the following diagram, which commutes by passing to the adjoint and then appealing to Observation \ref{obs:4.1}:
    \[
    \xymatrix{
        \overline{\OM^{\infty}_{\GA} E} = B(\GA,\GA,\OM^{\infty}E) \ar[dr]_{\ze} \ar[rr]^-{\ga} & & \OM^{\infty} B(\SI^{\infty},\GA,\OM^{\infty}E) = \OM^{\infty} \hat{E} \ar[dl]^{\OM^{\infty}\xi}\\
        & \OM^{\infty}E & 
    }
    \]
    \noindent where $\ga$ and $\ze$ are weak equivalences (by \textbf{\ref{salist:SA6}} and Lemma \ref{lem:4.1}). It follows by the two-out-of-three property that $\xi$ is a weak equivalence as well.
\end{proof}
\noindent Just as we had thought of $\overline{Y}$ as a derived resolution of $Y \in \GA[\sT]$, we should think of $\hat{E}$ as a derived connective cover of $E$ whose $\oU$th space is weak equivalent to the derived resolution of the $\oU$th space of $E$.\\

Piecing together the derived language from above, we get our first main result---a crude step towards homotopical monadicity.

\maintheorem*


\begin{proof}
    The functors $\bE$ and Bar have already been discussed. We note that $\OM^{\infty}_{\GA}$ and $\bE$ preserve weak equivalences since $\OM^{\infty}$ and $\SI^{\infty}$ preserve weak equivalences, and $|-|$ preserves weak equivalences between Reedy cofibrant spectra (this is \textbf{\ref{salist:SA3}}). Lemma \ref{lem:4.3} provides a natural weak equivalence $\xi \colon \hat{E} \rtarr E$ for every connective spectrum $E$.
    For $Y \in \GA[\sT]$, the homotopy equivalence $\ze \colon \overline{Y} \rtarr Y$ is given by Lemma \ref{lem:4.1}, and the weak equivalence $\ga \colon \overline{Y} \rtarr \OM^{\infty}_{\GA} \bE Y$ is furnished by \textbf{\ref{salist:SA6}}.
\end{proof}

The conceptual gap described following Lemma \ref{lem:4.1} is much more apparent now. We would really like to say that $\overline{Y}$ is a $\GA$-algebra such that $\ze$ and $\ga$ are $\GA$-algebra maps. As a result, it is a priori not clear that Theorem \ref{thm:0.1} provides an equivalence between the homotopy categories of $\GA$-algebras and connective spectra.
Nonetheless, with some conceptual investigation into the relation between $\GA$-algebras and simplicial $\GA$-algebras, we will establish a sharpened result.\\

\subsection{Op-lax maps of monads and abstract linking theory}
To further understand the relation between $\GA$-algebras and simplicial $\GA$-algebras, it will be helpful to think of simplicial $\GA$-algebras as themselves algebras over the simplicial adjunction monad $\GA_*$. The natural map $\ga : |\GA_*| \rtarr \GA |-|$ then relates these two monads---along with the data of $|-|$, it is an instance of an \emph{op-lax map} (see Proposition \ref{prop:linking}). In this section, 
we develop an abstract theory of \emph{linking} to understand how a general op-lax map between two monads can link their algebras.\\

We first recall the definition of an op-lax map between monads.

\begin{defn}\label{def:oplax}
    Let $(\bC,\mu,\et)$ and $(\bD,\nu,\de)$ be monads on categories $\sV$ and $\sW$ respectively. We say that a pair $(F,\al)$ with a functor $F \colon \sV \rtarr \sW$ and natural transformation $\al \colon F \bC \rtarr \bD F$ is an op-lax map of monads, or simply an op-lax map, if the following commute for every $X$.
    \[
    \xymatrix{
        F \bC \bC X \ar[r]^{\al_{\bC X}} \ar[d]_{F\mu_X} & \bD F \bC X \ar[r]^{\bD \al_X} & \bD \bD F X \ar[d]^{\nu_{FX}} \\
        F \bC X \ar[rr]_{\al_X} && \bD F X
    } \qquad 
    \xymatrix{
        & FX \ar[dl]_{F\et_X} \ar[dr]^{\de_{FX}} &\\
        F \bC X \ar[rr]_{\al_X} && \bD F X
    }
    \]
    We will find it customary to write $(F,\al) \colon \bC \rtarr \bD$.
\end{defn}
The following is the central definition to our linking theory. 

\begin{defn}\label{def:link}
Let $(\bC,\mu,\et)$ and $(\bD,\nu,\de)$ be monads on $\sV$ and $\sW$. Let $(F,\al) \colon \bC \rtarr \bD$.
We say that algebras $(X,\theta) \in \bC[\sV]$ and $(Y,\phz) \in \bD[\sW]$ are $F$-linked by $\be \colon FX \rtarr Y$ (called an $F$-linking map) if the following diagram commutes.
\[
\xymatrix{
    F \bC X \ar[d]_{F\theta} \ar[r]^{\al_X} & \bD F X \ar[r]^{\bD \be} & \bD Y \ar[d]^{\phz}\\
    FX \ar[rr]_{\be} && Y
}
\]
\noindent When there is no ambiguity in the op-lax map $(F,\al)$, we will just say that $X$ and $Y$ are linked, and that they are linked by the
map $\be$. 
\end{defn}

\begin{exmp}\label{rem:Gid}
    Consider the case where $\sV = \sW$, $\bC = \bD$, and the op-lax map of monads is $(\id,\id)$ on $\bC$. See that two algebras $(X,\theta)$ and $(Y,\phz)$ in $\bC[\sV]$
    are linked by $\be \colon X \rtarr Y$ iff the following commutes.
    \[
    \xymatrix{
        \bC X \ar[d]_{\theta} \ar[r]^{\bC \be} & \bC Y \ar[d]^{\phz}\\
        X \ar[r]_{\be} & Y
    }
    \]
    \noindent That is, iff $\be$ is a map of $\bC$-algebras! This tells us that maps of $\bC$-algebras are already linking maps. And more precisely that in the simplest case, linking maps should just be algebra maps. It is in this sense that
    we see our linking maps as generalizing maps of algebras.
\end{exmp}

\begin{exmp}
    We illustrate an elementary example of the idea of linking. Every group homomorphism $\la \colon M \rtarr N$ induces an op-lax map of monads from $M \times -$ to $N \times -$ on the category of sets. 
    This op-lax map is given by $(\id,\la \times \id)$. An $\id$-linking map $\be \colon X \rtarr Y$ from an $M$-set $X$ to an $N$-set $Y$ is the same thing as a $\la$-equivariant map, i.e. $\be(mx) = \la(m)\be(x)$ for all $m \in M$ and $x \in X$.
\end{exmp}
We now compare our definition of linking with the definitions of \cite{KMZ25}.
\begin{rem}\label{rem:KMZ25}
    Consider an op-lax map $(F,\al)$ from $\bC$ to $\bD$. Suppose we have $(X,\theta) \in \bC[\sV]$, $(FX,\psi) \in \bD[\sW]$, and $(Y,\phz) \in \bD[\sW]$. A map $\be \colon FX \rtarr Y$ is a linking map iff the
    following commutes.
    \[
    \xymatrix{
        F \bC X \ar[d]_{F\theta} \ar[r]^{\al_X} & \bD FX \ar[r]^{\bD \be} & \bD Y \ar[d]^{\phz} \\
        FX \ar[rr]_{\be} && Y
    }
    \]
    Note that if $\beta$ is a map of $\bD$-algebras, and the following triangle commutes, then $\be$ is a linking map.
    \[
    \xymatrix{
        F \bC X \ar[r]^{\al_X} \ar[d]_{F\theta} & \bD F X \ar[dl]^{\psi} \\
        FX
    }
    \]
    \noindent Note that this triangle commutes iff $\id \colon FX \rtarr FX$ is a linking map linking the algebras $(X,\theta)$ and $(FX,\psi)$. Whenever this is the case, \cite[Definition 2.31]{KMZ25} calls $X$ and $FX$ linked. As a consequence of this remark, we obtain that $X$ and $FX$ are linked in the sense of \cite{KMZ25} iff
    $X$ and $FX$ are linked via $\id \colon FX \rtarr FX$ in the sense of Definition \ref{def:link}. Nonetheless, the more abstract definition provided in \cite[Definition 4.2]{KMZ25} is equivalent to ours. 
\end{rem}

The following visual explanation concretizes what it means for algebras to be linked in our sense. Let us fix an op-lax map $(F,\al) \colon \bC \rtarr \bD$, algebras $(X,\tha) \in \bC[\sV]$ and $(Y,\phz) \in \bD[\sW]$, and a linking map $\be \colon FX \rtarr Y$. Consider the unit diagrams for the 
algebras $(X,\theta)$ and $(Y,\phz)$.
\[
\xymatrix{
    X \ar[r]^{\et_X} \ar@{=}[dr] & \bC X \ar[d]^{\theta}\\
    & X
} \qquad
\xymatrix{
    Y \ar[r]^{\de_Y} \ar@{=}[dr] & \bD Y \ar[d]^{\phz}\\
    & Y
}
\]
\noindent Applying $F$ to the unit diagram for $X$, the map $\be \colon FX \rtarr Y$ links together the unit diagrams to produce the commutative diagram below.
\[
\xymatrix@-0.25pc{
    & & Y \ar[r]^-{\de_Y} \ar@{=}[dr] & \bD Y \ar[d]^{\phz}\\
    & & & Y \\
    FX \ar[uurr]^{\be} \ar[r]_-{F\et_X} \ar@{=}[dr] & F \bC X \ar[uurr]_{\bD \be \circ \al_X} \ar[d]^-{F\theta} & &\\
    & FX \ar[uurr]_{\be} & &
}
\]
\noindent To see that this prism commutes, it suffices to see that its top face commutes. We can write the top face as the composite:
\[
\xymatrix{
    FX \ar[d]_{F\et_X} \ar[rr]^{\be} \ar[dr]^{\de_{FX}} & & Y \ar[d]^{\de_Y} \\
    F \bC X \ar[r]_{\al_X} & \bD FX \ar[r]_{\bD \be} & \bD Y
}
\]
\noindent The triangle in this diagram commutes since it is the unit diagram of the op-lax map $(F,\al)$, and the remaining trapezoid is a naturality square for $\de$.\\

Similarly, we have the product diagrams for $X$ and $Y$:
\[
\xymatrix{
    \bC \bC X \ar[r]^{\bC \theta} \ar[d]_{\mu_X} & \bC X \ar[d]^{\theta}\\
    \bC X \ar[r]_{\theta} & X
}
\qquad
\xymatrix{
    \bD \bD Y \ar[r]^{\bD \phz} \ar[d]_{\nu_Y} & \bD Y \ar[d]^{\phz}\\
    \bD Y \ar[r]_{\phz} & Y
}
\]
\noindent Once more, applying $F$ to the product diagram for $X$, the map $\be \colon FX \rtarr Y$ links together the product diagrams to produce the commuting diagram below.
\[
\xymatrix{
    & & & & \bD \bD Y \ar[r]^{\bD \phz} \ar[d]_{\nu_Y} & \bD Y \ar[d]^{\phz}\\
    & & \bD \bD FX \ar[urr]^{\bD \bD \be} & & \bD Y \ar[r]_{\phz} & Y\\
    F \bC \bC X \ar[urr]^{\bD \al_X \circ \al_{\bC X} \; \; \; \;} \ar[r]_{F\bC \theta} \ar[d]_{F\mu_X} & F \bC X \ar[uurrrr]_{\; \; \bD \be \circ \al_X} \ar[d]^{F\theta} &&&&\\
    F \bC X \ar[r]_{F\theta} \ar[uurrrr]_{\; \; \bD \be \circ \al_X} & FX \ar[uurrrr]_{\beta} &&&&
}
\]
\noindent The front and back faces of this cuboid commute since they are the product diagrams for $X$ and $Y$. The right face and the bottom face also commute since $X$ and $Y$ are linked by $\be$ by assumption.
We can decompose the top face into the following commutative diagrams:
\[
\xymatrix{
    F \bC \bC X \ar[r]^{\al_{\bC X}} \ar[d]_{F\bC \theta} & \bD F \bC X \ar[r]^{\bD \al_X} \ar[d]^{\bD F \theta} & \bD \bD FX \ar[r]^{\bD \bD \be} & \bD \bD Y \ar[d]^{\bD \phz}\\
    F \bC X \ar[r]_{\al_X} & \bD FX \ar[rr]_{\bD \be} & & \bD Y.
}
\]
\noindent The left square commutes as a naturality square for $\al$, and the right square commutes since it is $\bD$ applied to the linking diagram for $\be$. Similarly the left face of 
the cuboid also decomposes as:
\[
\xymatrix{
    F \bC \bC X \ar[r]^{\al_{\bC X}} \ar[d]_{F\mu_X} & \bD F \bC X \ar[r]^{\bD \al_X} & \bD \bD FX \ar[d]_{\nu_{FX}} \ar[r]^{\bD \bD \be} & \bD \bD Y \ar[d]^{\nu_Y}\\
    F \bC X \ar[rr]_{\al_X} & & \bD FX \ar[r]_{\bD \be} & \bD Y.
}
\]
\noindent The left square commutes as the product diagram for the op-lax map $(F,\al)$, and the right square commutes since it is the naturality square for $\nu$.\\

In essence, what we have tried to motivate here is this: to say the algebras $(X,\theta)$ and $(Y,\phz)$ are linked, is to say in a precise diagrammatic sense that their 
algebra structures are linked together via the linking map $\be$.\\

Examples \ref{ex:2.8} provides a general simplification of the kind that one may have hoped for in the loop space story, and indicates how it would have allowed us to evade the linking notion. The precise failure of the relevant hypothesis in the loop space story is what necessitates a general notion beyond maps of algebras.
\begin{exmp}\label{ex:2.8}
    Consider again an op-lax map $(F,\al) \colon \bC \rtarr \bD$ and algebras $(X,\tha) \in \bC[\sV]$ and $(Y,\phz) \in \bD[\sW]$. Suppose $\al_X \colon F \bC X \cong \bD F X$ is an isomorphism. There is then a canonical $\bD$-algebra structure on $FX$ given by $\ps = F \tha \circ \al_X^{-1}$. This definition satisfies the triangle condition from Remark \ref{rem:KMZ25} manifestly. We infer that a map $\be \colon FX \rtarr Y$
    is a linking map iff it is a map of $\bD$-algebras.
\end{exmp}

In what follows, we examine a general situation that is closer to our loop space story, with the idea of a \emph{section} of an op-lax map. Remark \ref{ex:3.10} provides a broad class of linking maps that we will use to sharpen our monadicity result in Section \ref{subsec:sharpen}.

\begin{defn}
    Let $(F,\al) \colon \bC \rtarr \bD$ be an op-lax map of monads from $(\bC,\mu,\et)$ (on $\sV$) to $(\bD,\nu,\de)$ (on $\sW$). We say that $(F,\al)$ admits a section 
    if there exists an op-lax map $(G,\id) \colon \bD \rtarr \bC$ such that $FG = \id$ and $\al_{GY} = \id$ for every $Y \in \sW$. Such an op-lax map is called a section of $(F,\al)$.
\end{defn}
\begin{rem}\label{ex:3.10}
    Let $(F,\al) \colon \bC \rtarr \bD$ be an op-lax map with section $(G,\id)$. If $(Y,\phz)$ is a $\bD$-algebra, then $(GY,G \phz)$ is automatically a $\bC$-algebra. For any other $\bD$-algebra $(Y',\phz')$, a map $\be \colon Y \rtarr Y'$ is a linking map iff it is a map of $\bD$-algebras.
    \[
    \xymatrix{
        \bD Y \ar@{=}[r] \ar[d]_{\phz} & F \bC GY \ar[d]_{\phz} \ar@{=}[r]^{\al_{GY}} & \bD FGY \ar[r]^{\bD \be} & \bD Y' \ar[d]^{\phz'} \\
        Y \ar@{=}[r] & FGY \ar[rr]_{\be} && Y'
    }
    \]
    Now consider algebras $(X,\tha) \in \bC[\sV]$ and $(Y,\phz) \in \bD[\sW]$. We show that if $f \colon X \rtarr GY$ is a map of $\bC$-algebras, then $\be := F(f) \colon FX \rtarr Y$ is a linking map. Since $f$ is a map of $\bC$-algebras, the following commutes.
    \[
    \xymatrix{
        F \bC X \ar[r]^{F \bC f} \ar[d]_{F\tha} & \bD Y \ar[d]^{\phz}\\
        FX \ar[r]_{Ff} & Y
    }
    \]
    Moreover by the naturality of $\al$ (applied to the morphism $f \colon X \rtarr GY$), we have the following commutative diagram.
    \[
    \xymatrix@-1pc{
        F \bC X \ar[rr]^{F \bC f} \ar[dr]_{\al_X} & & \bD Y\\
        & \bD F X \ar[ur]_{\bD F f} &
    }
    \]
    \noindent We conclude that the following commutes and therefore that $\be$ is a linking map.
    \[
    \xymatrix{
        F \bC X \ar[r]^{\al_X} \ar[d]_{F\tha} & \bD F Y \ar[r]^{\bD \be} & \bD Y \ar[d]^{\phz}\\
        FX \ar[rr]_{\be} & & Y
    }
    \]
    \noindent If furthermore $\al_X$ is an isomorphism, then by Example \ref{ex:2.8} there is a canonical $\bD$-algebra structure on $FX$, and $\be$ is necessarily a map of $\bD$-algebras.
    In applications, where this reduction is unavailable, the best we can say is that for $f \colon X \rtarr GY$ a map of $\bC$-algebras, $Ff$ is a linking map. In the loop space story, we will take $G$ to be the constant simplicial object functor $c_* \colon \sT \rtarr s\sT$. 

\end{rem}

Let us continue on the story of an op-lax map $(F,\al)$ from $\bC$ to $\bD$ with a section $(G,\id)$. Let $(Y',\phz')$ and $(Y,\phz)$ be $\bD$-algebras. We are now interested in maps $Y' \rtarr Y$ 
that are not algebra maps, but still retain some monadic structure. One might hope that the notion of linking that we have discussed provides the appropriate generalization. However, we saw in Remark \ref{ex:3.10},
that the only linking maps between $\bD$-algebras are the $\bD$-algebra maps! It seems more appropriate then to look at maps \emph{over} an algebra that is linked to both $Y$ and $Y'$. In other words, those maps that factor as spans of linking maps: $Y' \leftarrow FX \rightarrow Y$.
\begin{defn}\label{def:linked}
    Let $(F,\al) \colon \bC \rtarr \bD$ be an op-lax map of monads with section $(G,\id)$. We say that a map 
    $f \colon Y' \rtarr Y$ is a \emph{linked map} if there exist linking maps $\rho \colon FX \rtarr Y'$ and 
    $\be \colon FX \rtarr Y$ such that $\be = f \circ \rho$.
    \[
    \xymatrix@-1pc{
        & FX \ar[dl]_-{\rho} \ar[dr]^{\be} &\\
        Y' \ar[rr]_{f} && Y
    }
    \]
\end{defn}


\subsection{Simplicial $\GA$-algebras and sharpened homotopical monadicity}\label{subsec:sharpen}
Our discussion so far on abstract linking theory is instantiated here to provide a sharpened version of the homotopical monadicity theorem.
From the adjunction $(\SI^{\infty},\OM^{\infty})$ we get an induced adjunction $(\SI^{\infty}_*,\OM^{\infty}_{*})$ between the categories of simplicial objects such that the simplicial unit $\et_*$ and counit $\epz_*$ are just levelwise applications of $\et$ and $\epz$.
The associated simplicial adjunction monad is given by $\GA_*$ on $\sT$, and consequently its product map $\mu_*$ is just a levelwise application of $\mu$. As remarked earlier, by \emph{simplicial} $\GA$\emph{-algebras} we will mean $\GA_*$-algebras. Once more, by the usual yoga of adjunction monads, 
$\OM^{\infty}_* E_*$ has a canonical simplicial $\GA$-algebra structure and therefore $\OM^{\infty}_*$ really lands in the category $\GA_*[s\sT]$. When we want to emphasize this fact, we use the notation $\OM^{\infty}_{*,\GA}$ for the functor landing in $\GA_*[s\sT]$. Throughout this section, we will also conflate $\ga$ with the natural map $\ga \colon |\OM^{\infty}_* \SI^{\infty}_*| \rtarr \OM^{\infty} |\SI^{\infty}_*| \cong \OM^{\infty} \SI^{\infty} |-|$.\\

The following proposition outlines the setting in which we will apply our linking theory.


\begin{prop}\label{prop:linking}
    $(|-|,\ga)$ is an op-lax map of monads $\GA_* \rtarr \GA$. That is to say, $|-| \colon s\sT \rtarr \sT$ is equipped with a natural transformation $\ga \colon |\GA_*| \rtarr \GA|-|$ such that the following commute.
    \[
    \xymatrix{
        |\GA_* \GA_*| \ar[d]_{|\mu_*|} \ar[r]^{\ga \GA_*} & \GA |\GA_*| \ar[r]^-{\GA \ga} & \GA \GA |-| \ar[d]^{\mu |-|} \\
        |\GA_*| \ar[rr]_{\ga} && \GA|-|
    }
    \qquad \qquad
    \xymatrix{
        & |-| \ar[dl]_{|\et_*|}  \ar[dr]^{\et|-|} &\\
        |\GA_*| \ar[rr]_{\ga} & & \GA|-|
    }
    \]
    Moreover for every $E_* \in s\sS$, the algebras $\OM^{\infty}_{*,\GA} E_*$ and $\OM^{\infty}_{\GA}|E_*|$ are $|-|$-linked by $\ga$, meaning the following diagram commutes.
    \[
    \xymatrix{
        |\GA_* \OM^{\infty}_* E_*| \ar[d]_{|\OM^{\infty}_*\epz_*|} \ar[r]^-{\ga \OM^{\infty}_*} & \GA |\OM^{\infty}_* E_*| \ar[r]^{\GA \ga} & \GA \OM^{\infty} |E_*| \ar[d]^{\OM^{\infty} \epz |-|} \\
        |\OM^{\infty}_* E_*| \ar[rr]_{\ga} && \OM^{\infty}|E_*|
    }
    \]
\end{prop}
\begin{proof}
    We prove the commutativity of the unit diagram for $(|-|,\ga)$ (the top right diagram), and the diagram for the linking of $\OM^{\infty}_*$ to $\OM^{\infty}$---the remaining diagram follows automatically.
    For the unit diagram, we first pass to the adjoint:
    \[
    \xymatrix{
        & |\SI^{\infty}_*| \ar[dl]_{|\SI^{\infty}_* \et_*|} \ar[dr]^{|\SI^{\infty}_*|} &\\
        |\SI^{\infty}_* \OM^{\infty}_* \SI^{\infty}_*| \ar[rr]_{|\epz_* \SI^{\infty}_*|} && |\SI^{\infty}_*|.
    }
    \]
    \noindent This diagram is just the realization of a triangle identity for the adjunction $(\SI^{\infty}_*,\OM^{\infty}_*)$ in $s\sT$, and therefore it commutes. As for the product diagram, passing to the adjoint again, we have the slightly larger diagram (where $\tilde{\ga} \colon \SI^{\infty}|\OM^{\infty}| \rtarr |-|$ denotes the adjoint map to $\ga$):
    \[
    \xymatrix{
        \SI^{\infty}|\GA_* \OM^{\infty}_*| \ar[r]^-{\cong} \ar[d]_{\SI^{\infty}|\OM^{\infty}_* \epz|} & |\SI^{\infty}_* \OM^{\infty}_* \SI^{\infty}_* \OM^{\infty}_*| \ar[r]^-{|\epz_*|} \ar[d]_{|\SI^{\infty}_* \OM^{\infty}_* \SI^{\infty}_*|} & |\SI^{\infty}_* \OM^{\infty}_*| \ar[d]_{|\epz_*|} \ar[r]^{\cong} & \SI^{\infty}|\OM^{\infty}_*| \ar[r]^-{\SI^{\infty}\ga} \ar[dl]^{\tilde{\ga}} \ar[dr]_{\tilde{\ga}} & \SI^{\infty} \OM^{\infty}|-| \ar[d]^{\epz|-|} \\
        \SI^{\infty}|\OM^{\infty}_*| \ar[r]_{\cong} & |\SI^{\infty}_* \OM^{\infty}_*| \ar[r]_{|\epz_*|} & |-| \ar[rr]_{=} && |-|
    }
    \]
    The left two squares commmute trivially. The left triangle in the rightmost square commutes by definition of $\ga$, and the right triangle commutes since $\ga$ and $\tilde{\ga}$ are adjoint maps related by the counit $\epz$. 
\end{proof}

As a consequence, we see that for each $Y \in \GA [\sT]$, $\ol{Y}$ and $\OM^{\infty}_{\GA} \bE Y$ are linked by $\ga \colon \ol{Y} \rtarr \OM^{\infty}_{\GA} \bE Y$. Now we can make a remark on \emph{what could have been}, using Example \ref{ex:2.8}. If for some simplicial $\GA$-algebra $X_*$ we had $\ga \colon |\GA_* X_*| \cong \GA |X_*|$, then by Example \ref{ex:2.8} the realization $|X_*|$ would have a canonical $\GA$-algebra structure, and linking maps out of $|X_*|$ would just be maps of $\GA$-algebras.
In particular, this is maybe what one would have hoped for $\OM^{\infty}_{*,\GA} \bE_* Y$ and $\ol{Y}$.

\begin{prop}
    For every $X \in \sT$, $c_* \GA X = \GA_* c_* X$. Consequently, the pair $(c_*,\id)$ defines a section of the op-lax map $(|-|,\ga)$.
\end{prop}
\begin{proof}
The first statement and the fact that $(c_*,\id)$ provides an op-lax map $\GA \rtarr \GA_*$ both follow from the levelwise definitions of $\GA_*$, $\mu_*$, and $\et_*$. By \textbf{\ref{salist:SA1}}, for every $X \in \sT$, $|c_* X| \cong X$. It remains to note that $\ga_{c_* X} \colon |\GA_* c_* X| = |c_* \GA X| \rtarr \GA X$ is the identity map $\id_{\GA X}$. $\ga_{c_* X}$ was defined to be the adjoint to the map $|(\epz_*)_{c_* \SI^{\infty}X}|$. By definition again, $|(\epz_*)_{c_* \SI^{\infty}X}| = \epz_{\SI^{\infty}X}$. However $\epz_{\SI^{\infty} X}$
was defined to be the adjoint to $\id_{\GA X}$, forcing $\ga_{c_* X} = \id_{\GA X}$. 
\end{proof}

Remark \ref{ex:3.10} has several important consequences in light of the previous proposition. First, as was already intuitively clear, if $X$ is a $\GA$-algebra, then there is a canonical $\GA_*$-algebra structure on $c_* X$ given by $\GA_* c_* X = c_* \GA X \xrightarrow{c_* \tha} c_* X$. Moreover, linking maps between $\GA$-algebras, $X \rtarr X'$ (seen as maps $|c_* X| \rtarr X'$) are exactly the same as maps of $\GA$-algebras. 
Now suppose for $X_* \in \GA_* [\sT]$ and $X' \in \GA [\sT]$, we had a map of $\GA_*$-algebras $f_* \colon X_* \rtarr c_* X'$. If it had been the case that $\ga \colon |\GA_* X_*| \cong \GA |X_*|$, then $|f_*|$ would have been a $\GA$-algebra map. Nonetheless, Remark \ref{ex:3.10} shows that $|f_*| \colon |X_*| \rtarr X'$ is always a linking map. In particular, $\ze := |\ze_*|$ is a linking map! In our homotopical setting, Definition \ref{def:linked} specializes to provide the notion of \emph{linked weak equivalences}. These are the linked maps that are also weak equivalences.\\

What we have put together over this section and the last is the second main result of this paper, a sharpening of Theorem \ref{thm:0.1}.

\maintheoreM*

The consequence for the homotopy categories of $\GA$-algebras and connective spectra that one would hope for, is provided by a direct corollary of this sharpened result.

\maincor*

\section{Verification of the simplicial axioms}\label{sec:simpob}
\subsection{The easy axioms}\label{subsec:simpobez}
In this section we prove the simplicial axioms \textbf{\ref{salist:SA1}-\ref{salist:SA5}} from Section \ref{sec:prelim}. We have denoted by $s\aC$ the category of simplicial objects in a category $\aC$. In particular, $s\sT = \ssSh((\mathrm{Sm}_S)_{Nis})$, the category of bisimplicial sheaves on $\mathrm{Sm}_S$. In order to prescribe realization functors $|-| \colon s\sT \rtarr \sT$ and $|-| \colon s\sS \rtarr \sS$, we must first record our choice of the cosimplicial objects in $\sT$ and $\sS$. There is a conventional choice in the case of $\sT$. For each $\bn \in \DE$, let $\DE_n$ denote the object in $\sT$ representing $S \times_{\Spec \bZ} \Spec \bZ[x_0, \dots, x_n] / \left(\sum_i x_i = 1\right)$. Given $f \colon \bn \rtarr \bm$ in $\DE$, there is a map $\DE_f \colon \DE_n \rtarr \DE_m$ (considered as schemes) given by the ring homomorphism $\DE_f(x_i) = \sum_{j \in f^{-1}(i)}x_j$ if $f^{-1}(i) \neq \emptyset$ and $0$ otherwise. Together, this defines a cosimplicial object in $\sT$ such that each $\DE_n$ is isomorphic to $\bA^n_S$. The typical prescription, noting that $\sT$ is tensored over itself, then leads us to define $|K_*| = K_* \otimes_{\DE} \DE_*$ for $K_* \in s\sT$. The realization functor here however takes a much simpler form due to the bisimplicial nature of simplicial spaces. For $K_* \in s\sT$, thought of a bisimplicial sheaf $K_{*,*}$, we have that $|K_*|$ is isomorphic to the diagonal simplicial sheaf $d(K)_n = K_{n,n}$. By the usual yoga, since $\sT$ is cotensored over itself, $|-| \colon s\sT \rtarr \sT$ is a left adjoint functor with right adjoint $\bS \colon \sT \rtarr s\sT$ given by $\bS(X)_q = F(\DE_q,X)$ where $F(-,-)$ is the cotensor functor $\sT^{\mathrm{op}} \times \sT \rtarr \sT$. Since $\sS$ is tensored and cotensored over $\sT$, our choice of the cosimplicial objects in $\sT$ yields the realization functor $|-| \colon s\sS \rtarr \sS$ given by $|E_*| = E_* \otimes_{\DE} \DE_*$ for simplicial spectra and its right adjoint (which we also denote $\bS \colon \sS \rtarr s\sS$) given by $\bS(E)_q = F(\DE_q,E)$. As in the case of spaces, $F$ is used to denote the cotensor functor $\sT^{op} \times \sS \rtarr \sS$.
This proves \textbf{\ref{salist:SA1}}.\\

The following is \textbf{\ref{salist:SA2}}, where the notion of a simplicial homotopy is that of \cite[Definitions 5.1]{may67simpob}. 
\begin{prop}\label{prop:3.1}
If $h_*$ is a homotopy between simplicial maps $f_*, g_* \colon K_* \rtarr L_*$ in $s\sT$, then $|h_*|$ determines a homotopy between maps $|f_*|, |g_*| \colon |K_*| \rtarr |L_*|$ in $\sT$. The corresponding statement for $\sS$ also holds true.
\end{prop}
\begin{proof}
It is important to first note that $|-|$ preserves finite products since $\sT$ (likewise $\sS$) is cartesian closed and since $|-|$ preserves products of representables. The proof of this fact is purely formal, starting from the co-Yoneda lemma. By $\DE[1]$, let us denote the standard simplicial set $\bn \mapsto \Hom_{\DE}(\bn,\mathbf{1})$, considered a simplicial space using the constant scheme construction of \cite[Exposé I, 1.8]{SGA3T1} levelwise. In the case of $\sT$, by \cite[Proposition 6.2]{may67simpob} there exists $H_* \colon \DE[1] \times K_* \rtarr L_*$ defined as follows. Working over $U \in \mathrm{Sm}_S$, we let $x \in K_*(U)$ and define $H_*((0),x)=g_*(x)$, $H_*((1),x) = f_*(x)$ (here $(0)$ denotes any simplex of $\pa_1\DE[1]$ and $(1)$ any simplex of $\pa_0\DE[1]$), and $H_q(s_{q-1} \cdots s_{i-1} s_{i+1} \cdots s_0 \mathbf{1},x) = \pa_{i+1}h_i(x)$ for each $0 \leq i \leq q-1$. $H_*$ constructed in this way is not only a map of simplicial sets but also one of simplicial spaces since $h_i$ and $\pa_{i+1}$ are morphisms in $\sT$. Finally, it remains to note that $|\DE[1]|$ is isomorphic to $\bA^1$, so that the composite $\bA^1 \times |K_*| \xrightarrow{\sim} |\DE[1]| \times |K_*| \xrightarrow{\sim} |\DE[1] \times K_*| \xrightarrow{|H_*|} |L_*|$ provides the required homotopy. The story for spectra is similar.\footnote{The proof here is almost identical to that of \cite[Corollary 11.10]{May72}, and might be formalized to a more general context--for example to a cartesian closed concrete category with a notion of homotopy.}
\end{proof}

Since $\SI^{\infty}$ is a left adjoint functor, it commutes with the colimits that build $|-| \colon s\sT \rtarr \sT$. Thanks to this observation, we obtain \textbf{\ref{salist:SA4}}.

\begin{prop}\label{prop:3.2}
There exists a natural isomorphism $\SI^{\infty}|K_*| \iso |\SI^{\infty}_*K_*|$ for each $K_* \in s\sT$.
\end{prop}

We now proceed toward \textbf{\ref{salist:SA3}}. Recall that a simplicial object $X_* \in s\sV$ for a category $\sV$ is called \emph{Reedy cofibrant} if the latching maps $LX_q = \colim_{\phi \colon \mathbf{q+1} \epito \mathbf{s}, \phi \neq \id} X_s \monoto X_{q+1}$ are cofibrations for each $q$. $LX_q$ here is the scheme-theoretic analog to the union of all degenerate $q$-simplices in the simplicial set context. As in \cite{May09Einfty} we implicitly assume that our spaces are nondegenerately based (i.e. that for each based space $X$, the basepoint inclusion $S \rtarr X$ is a cofibration). So in this case every simplicial space is Reedy cofibrant.
\textbf{\ref{salist:SA3}} is derived from a general model-theoretic fact first due to \cite{reedy73}. The presentation here is inspired from \cite{may74perm}. 

\begin{prop}
$|-|$ preserves weak equivalences between Reedy cofibrant objects in $s\sT$ (resp. in $s\sS$).
\end{prop}

We define weak equivalences in $s\sT$ and $s\sS$ as those maps that are level-wise weak equivalences (in the appropriate model category)\footnote{In a sense, this result records the extent to which $|-|$ induces an equivalence between the homotopy categories of $s\sT$ and $\sT$ (resp. $s\sS$ and $\sS$).}. The proof of the above theorem relies on the motivic analog of Brown's gluing theorem \cite[Theorem 7.5.7]{Brown06topgrp}.

\begin{thm}\label{thm:3.1}
Suppose given a commutative diagram in $\sT$:
\[\xymatrix@-0.5pc{
& X \ar[rr]^g \ar '[d]^j [dd]^<f & & Z \ar[dd]^{\overline{j}}\\
A \ar[dd]_{i} \ar[rr] \ar[ur]^{\alpha} &  & C \ar[dd] \ar[ur]_{\beta} &\\
& Y \ar '[r]_{\overline{g}} [rr]^<{\overline{i}} & & W\\
B \ar[ur]^{\ga} \ar[rr]_{\overline{f}} & & D \ar[ur]_{\delta} &
}\]
\noindent that $i$ and $j$ are cofibrations, and that the front and back squares are pushouts. If $\alpha$, $\beta$, and $\gamma$ are homotopy equivalences, then so is $\delta$.
\end{thm}

\begin{proof}[Proof of Proposition 3.3] As before, we provide the proof for spaces---the proof for spectra is nearly identical and is hence omitted. Let $f_* \colon X_* \rtarr X'_*$ be a map of Reedy cofibrant spaces. We have the following pushout square for each $q$:
\[
\xymatrix{
(LX_q \times \DE_{q+1}) \sqcup_{LX_q \times \pa \DE_{q+1}} (X_{q+1} \times \DE_{q+1}) \ar[r]^-g \ar[d] & F_q|X_*| \ar[d]\\
X_{q+1} \times \pa \DE_{q+1} \ar[r] & F_{q+1}|X_*|.
}
\]
\noindent Here the standard map $g(s_i x, u) = [x,\si_i u]$ and $g(x,\de_i v) = [\pa_i x, v]$ is used. The prescription of the map here is interpreted to be given at an arbitrary object in the site $U \in \mathrm{Sm}_S$. $F_q|X_*|$ denotes the $q$th piece of the natural filtration on $|X_*|$, given by capping off the tensor product defining $|X_*|$ at the $q$-simplices. Alternatively, one might take the above to be the definition of $F_q|X_*|$ inductively (with $F_0|X_*|=X_0$) so that $|X_*| = \varinjlim F_q |X_*|$. There is also a similar square for $X'_*$ of course. Proceeding inductively and applying Theorem \ref{thm:3.1}, we see it suffices to show that $f_q \colon LX_q \rtarr LX'_q$ is a homotopy equivalence for each $q$. For each $0 \leq k \leq q$ let $L^k X_q = \colim_{\phi \colon \mathbf{q+1} \epito \mathbf{s}, s \leq k} X_s$. We obtain the following commutative diagram with a pushout square on the right side: \[\xymatrix{
L^{k-1} X_{q-1} \ar[d] \ar[r]^-{L_k}_-{\sim} & L^{k-1}X_{q} \times_{X_{q+1}} L_k X_q \ar[r] \ar[d] & L^{k-1}X_q \ar[d]\\
X_q \ar[r]^{L_k}_{\sim} & L_k X_q \ar[r] & L^k X_q
}\]
\noindent Here $L_k X_q$ denotes the image of $X_k \monoto X_{q+1}$, the map seen in the colimit definition of $LX_q$. Proceeding by induction on $q$, assuming that each $L^{k-1}X_{q-1} \rtarr L^{k}X_{q-1}$ is a cofibration for each $0 < k < q$, we conclude that $L^{k-1}X_q \rtarr L^{k}X_q$ is a cofibration. The same is true of $X'$. Since we know $L_0 \colon X_q \rtarr L_0 X_q$ is an isomorphism, $f_{q+1} \colon L^0 X_q \rtarr L^0 X'_q$ must be a homotopy equivalence. By induction on $q$, and induction on $k$ having fixed $q$, and Theorem \ref{thm:3.1}, we know that $f_{q+1} \colon L^k X_q \rtarr L^k X'_q$ is a homotopy equivalence for all $k$ and $q$.  
\end{proof}

\textbf{\ref{salist:SA5}} is clear from Proposition \ref{prop:3.2}, our definition of connective spectra, and the fact that $|-|$ is a left adjoint. Alternatively, if one defines homotopy sheaves of spectra as being certain direct limits of homotopy sheaves of spaces, then since $|-|$ is a left adjoint, it commutes with those direct limits. The statement is clear once again if one takes connective spectra to be spectra with certain vanishing homotopy sheaves. 
\begin{prop}
$|-| \colon s\sS \rtarr \sS$ takes levelwise connective spectra to connective spectra.
\end{prop}

\subsection{The hard axiom}\label{subsec:simpobhard}
We conclude with the proof of \textbf{\ref{salist:SA6}}, the most significant assumption in this paper. 
\begin{prop}\label{prop:3.5}
The map $\ga \colon |\OM^{\infty}_* E_*| \rtarr \OM^{\infty}|E_*|$ given by the adjoint to the composite $\SI^{\infty}|\OM^{\infty}_*E_*| \iso |\SI^{\infty}_* \OM^{\infty}_* E_*| \xrightarrow{|\epz|} |E_*|$ is, for levelwise connective $E_*$, a weak equivalence in $\sT$.
\end{prop}

An important distinction from the theory of operadic monads must be made here. Generally speaking, a monad $\bC$ coming from an operad commutes with $|-|$ up to isomorphism by the categorical Fubini theorem. For the adjunction monad however, as we will show now, a priori $\OM^{\infty}$ commutes with $|-|$ only up to weak equivalence---and therefore $\GA$ commutes with $|-|$ only up to weak equivalence.\\

We prove Proposition \ref{prop:3.5} by reducing to $n$-fold and singlefold loop spaces first, and then passing to the colimit. For $X \in \sT$ denote by $\OM X$ the hom-space $\underline{\Hom}_{\sT}(S^{\bA^1},X)$ (where $S^{\bA^1} = \bA^1/(\bA^1 - \{0\})$). Denote by $\OM^n X$ the hom-space $\underline{\Hom}_{\sT}(S^{\bA^n},X)$ (where $S^{\bA^n} = \bA^n/(\bA^n-\{0\})$. $\OM^n$ (as seen by the loop-suspension adjunction) is the $n$-fold composition of $\OM$ with itself. There is a map $|\OM^n_* X_*| \rtarr \OM^n|X_*|$ again given by the adjoint to $\SI^n|\OM^n_* X_*| \iso |\SI^n_*\OM^n_* X_*| \xrightarrow{|\epz|} |X_*|$ (here $\SI^n(-) = - \wedge S^{\bA^n}$) which we also call $\gamma$.

\begin{lem}\label{lem:3.1}
If $X_* \in s\sT$ is Reedy cofibrant and each $X_q$ is connected, then $\gamma \colon |\OM_* X_*| \rtarr \OM|X_*|$ is a weak equivalence in $\sT$.
\end{lem}

We first note that Lemma \ref{lem:3.1} allows us to conclude Proposition \ref{prop:3.5}. Of course, iteratively applying the above lemma tells us that $\ga^n \colon |\OM^n_* X_*| \rtarr \OM^n|X_*|$ is a weak equivalence for Reedy cofibrant $X_*$ with connected $X_q$. The idea now is that (under the assumptions on $E_*$) the map $\ga \colon |\OM^{\infty}_* E_*| \rtarr \OM^{\infty}|E_*|$ can be identified with the colimit of the maps $\ga^n \colon |\OM^n_* (E_n)_*| \rtarr \OM^n|(E_n)_*|$, so that Proposition \ref{prop:3.5} follows by passing to colimits. We must clarify the notation used. We have written $E_n$  to mean $E_{\oU/\bA^n}$---the \emph{sequential} coordinatized indexing here is chosen purely for simplicity. Note that each $(E_n)_*$ is a simplicial space, and $\OM^n_* (E_n)_*$ is compatibly isomorphic to $(E_0)_*$, so that the colimit of $|\OM^n_* (E_n)_*|$ is identified with $|\OM^{\infty}_*E_*|$. By a remark made after Proposition \ref{prop:2.1}, for $E_* = LT_*$, we can compute the colimit of $\OM^n |(E_n)_*|$ as $\OM^{\infty} |L T_*| \iso \OM^{\infty} L|T_*|$, where the realization of a simplicial prespectrum is defined as the levelwise realization--so that the needed observation follows by working at the prespectrum level.\\

There is, however, a caveat. The argument we use moving forward only works on fibrant objects---this requires us to pass to fibrant replacements. We introduce fibrant replacements in $s\sT$ after proving Lemma \ref{thm:3.2and} thus proving Lemma \ref{lem:3.1}. The following is a precise reformulation of the claim that emphasizes what is needed.

\begin{lem}\label{lem:3.1.2}
Let $X_* \in s\sT$, then $|P_*X_*|$ is contractible, and there are natural morphisms $\ga$ and $\de$ making the following diagram commute:
\[
\xymatrix{
|\OM_* X_*| \ar[r]^{\subset} \ar[d]_{\ga} & |P_*X_*| \ar[d]_{\de} \ar[r]^{|p_*|} & |X_*| \ar@{=}[d]\\
\OM |X_*| \ar[r]^{\subset} & P|X_*| \ar[r]^p & |X_*|.
}
\]
\noindent Moreover, when $X_*$ is fibrant, Reedy cofibrant, and level-wise connected, the top row of the above diagram $\xymatrix@1{|\OM_* X_*| \ar[r] & |P_*X_*| \ar[r]^{|p_*|} & |X_*|}$ is a homotopy fiber sequence, and therefore $\ga \colon |\OM_* X_*| \rtarr \OM|X_*|$ is a weak equivalence.
\end{lem}

\noindent We have used the standard notation $PY$ to denote the path-space of $Y \in \sT$, defined by $\underline{\Hom}_{\sT} (\bA^1_+, Y)$. The simplicial version $P_*Y_*$ for $Y_* \in s\sT$ is defined by taking the levelwise path-space. We outline the proof of the first part here. Firstly, we note that there is a natural contracting homotopy equivalence $h \colon PY \rtarr *$ given by $\phz \mapsto \phz \circ \iota_0$ (here $\iota_0 \colon S \rtarr \bA^1_S$ denotes the inclusion of $S$ as the basepoint 0 of $\bA^1_+$ and we identify the composite $S \rtarr Y$ with the basepoint of $Y$). This homotopy, applied to each $PX_q$, descends to a simplicial contracting homotopy $\DE[1] \times P_*X_* \rtarr P_*X_*$. By Proposition \ref{prop:3.1}, we conclude that $|P_*X_*|$ is indeed contractible. We are left with defining $\de \colon |P_*X_*| \rtarr P|X_*|$. Everything that follows in this paragraph fixes an arbitrary $U \in \mathrm{Sm}_S$, and we work over this $U$ implicitly. For $f \in PX_q$, $u \in \DE_q$, and $t \in \bA^1$, define $\de[f,u](t)=[f(t),u]$. $\de$ is a well-defined morphism in $\sT$, and it indeed makes the diagram in question commute. In some sense, this definition of $\de$ was the only sensible choice we had. There is categorical reason for this. The functors $- \wedge \bA^1_+$ and $P(-)$ are adjoint to each other. The map $\de \colon |P_*X_*| \rtarr P|X_*|$ is the adjoint to $|P_*X_*| \wedge \bA^1_+ \iso |(PX \wedge \bA^1_+)_*| \xrightarrow{|\epsilon_*|} |X_*|$ where $\epsilon \colon PX \wedge \bA^1_+ \rtarr X$ is the counit map.\\

What remains is the following, an analog of J.P. May's result in \cite{May72}, and a specialization of Anderson's result in \cite{anderson78fib}, for the geometric realization of fibrations.

\begin{lem}\label{thm:3.2and}
If $p_* \colon E_* \rtarr B_*$ is a fibration in $s\sT$ with fiber $F_*$ and $B_*$ is Reedy cofibrant and level-wise connected then $\xymatrix@1{|F_*| \ar[r] & |E_*| \ar[r]^{|p_*|} & |B_*|}$ is a homotopy fiber sequence. Moreover, when $X_*$ is fibrant, $p_* \colon P_*X_* \rtarr X_*$ is a fibration in $s\sT$ with fiber $\OM_* X_*$, leading us to conclude that $\xymatrix@1{|\OM_* X_*| \ar[r] & |P_*X_*| \ar[r]^{|p_*|} & |X_*|}$ is a homotopy fiber sequence (when $X_*$ is Reedy cofibrant and levelwise connected).
\end{lem}

We must first define what we mean by a fibration $p_* \colon E_* \rtarr B_*$ in $s\sT$.
This definition follows that of Reedy for the model structure on $s\aC$ given a model category $\aC$, using coskeletons. Namely a map $p_* \colon E_* \rtarr B_*$ is a fibration in $s\sT$ if for each $m$, the induced map $E_{m+1} \rtarr \mathrm{cosk}_m(E)_{m+1} \times_{\mathrm{cosk}_m(B)_{m+1}} B_{m+1}$ is a fibration in $\sT$. Together with the dually defined cofibrations in $s\sT$, we obtain a model structure on $s\sT$ that is in fact simplicially enriched. There is a pairing $- \ten - \colon \mathbf{sSet} \times s\sT \rtarr s\sT$ given by $(A_* \ten X_*)_k = X_k[A_k] := \coprod_{a \in A_k} X_k$. Appealing to the dual provides a pairing $\Hom (-,-) \colon \mathbf{sSet}^{\op} \times s\sT \rtarr s\sT$. Finally, there is the standard $\mathbf{sSet}$-enrichment on $s\sT$, $\Hom(-,-)\colon s\sT^{\op} \times s\sT \rtarr \mathbf{sSet}$, given by $\Hom(X_*,Y_*) = \Hom_{s\sT}(\DE[-] \ten X_*, Y_*) \in \mathbf{sSet}$. Moreover, these pairings satisfy the pullback-power axiom\footnote{They also satisfy an equivalent dual \emph{pushout-copower} axiom.}:

\begin{itemize}
\item If $i_* \colon A_* \rtarr B_*$ is a cofibration in $\mathbf{sSet}$ and $p_* \colon X_* \rtarr Y_*$ is a fibration in $s\sT$, then the following is a fibration in $\mathbf{sSet}$. \[\Hom(B_*,X_*) \rtarr \Hom(B_*,Y_*) \times_{\Hom(A_*,Y_*)} \Hom(A_*,X_*)\]  \noindent Moreover, the above map is an acyclic fibration if $p_*$ is also a weak equivalence.
\end{itemize}

With this new enriched model structure in our toolbox, we now handle the simplicial path space fibration on fibrant objects. That $p_* \colon P_*X_* \rtarr X_*$ is a fibration in $s\sT$ when $X_*$ is fibrant, is an instance of a well-known general fact (see \cite[Lemma 7.5]{goerssjardine} for example). The idea is to begin with the cofiber sequence of pointed simplicial sets\footnote{We consider $\DE[1]$ to be pointed at 1 here.} $\partial \DE[1] \monoto \DE[1] \rtarr \DE[1]/\partial \DE[1]$. The pullback-power axiom for $s\sT$ discussed above then ensures that when we power this sequence by a fibrant $X_*$, we obtain a fiber sequence $\Hom(\DE[1]/\partial \DE[1], X_*) \rtarr \Hom (\DE[1],X_*)\rtarr \Hom(\partial\DE[1],X_*)$. 
Noting that $\OM_* X_* \simeq \Hom(\DE[1]/\partial \DE[1],X_*)$, $P_*X_* \simeq \Hom(\DE[1],X_*)$, and $\Hom(\partial \DE[1],X_*) \simeq X_*$, this yields a fiber sequence $\OM X_* \rtarr PX_* \rtarr X_*$.\\

It remains to show that if $p_* \colon E_{*} \rtarr B_{*}$ is a fibration in $s\sT$ with fiber $F_*$ and $B_*$ is Reedy cofibrant and level-wise connected, then $|F_*| \rtarr |E_*| \rtarr |B_*|$ is a homotopy fiber sequence. This is the real meat of Lemma \ref{thm:3.2and}. In \cite{May72}, the author employs the Dold-Thom criteria for quasifibrations to achieve the required analogous result. 
Anderson, in \cite{anderson78fib}, provides a conceptually clearer treatment of the result by passing to bisimplicial sets and solving a lifting problem with respect to the image of the left adjoint of the realization functor.
The initial idea here was to make use of Anderson's argument, or the $\pi_*$-Kan condition for bisimplicial sets (due to Bousfield and Friedlander \cite{BF1978})---both of which, in a sense, rely on the construction of the simplicial groupoid $\PI_{\infty}(-)$\footnote{For completeness, the groupoid $\PI_{\infty} (Z)$ is the wreath product of the fundamental groupoid with the product of the higher homotopy groups of $Z$.}. However, the conceptual reformulation, and the strengthened result pertaining to geometric realizations of homotopy pullbacks, in \cite{Rezk14} is more amenable to our purposes---and is easier to approach. 
A specialized application of \cite[Proposition 5.5.6.17]{Lurie2017} to the $\infty$-topos $L_{W}\sT$ also yields the required result, but we expect that the generality of Lurie's statement will be more relevant in future studies where Rezk's argument is not readily accessible. We state the goal now:

\begin{thm}\label{thm:rezk}
    Let $B_* \in s\sT$ be a simplicial space such that the simplicial sheaf $\pi_0^{\bA^1}(B_*)$, defined levelwise as $\pi_0^{\bA^1}(B_*)_n = \pi_0^{\bA^1} (B_n)$, is discrete (as a simplicial object). Let $f_* \colon Y_* \rtarr B_*$ be a map in $s\sT$.
    Then for every homotopy pullback square, 
    \[
    \xymatrix{
        F_* \ar[r] \ar[d] & Y_* \ar[d]^{f_*}\\
        E_* \ar[r]_{p_*} & B_*
    }
    \]
    the following square,
    \[
    \xymatrix{
        |F_*| \ar[r] \ar[d] & |Y_*| \ar[d]^{|f_*|}\\
        |E_*| \ar[r]_{|p_*|} & |B_*|
    }
    \]
    obtained by applying $|-|$ to each corner, is still a homotopy pullback.
\end{thm}

\subsection{Realization-fibrations and the proof of Theorem \ref{thm:rezk}}\label{subsec:rf}

The notation used in Theorem \ref{thm:rezk} is of course suggestive of the fact that in our applications, we will take $Y_*$ to be the terminal simplicial space---so that the level-wise connectedness of $B_*$ will give us Lemma \ref{thm:3.2and}.
But one must take brief respite to make a conceptual remark about Theorem \ref{thm:rezk} to put us directly in the framework of \cite{Rezk14}. Implicitly restricting down to Reedy cofibrant objects, a routine check using the Bousfield-Kan map tells us that there is a weak equivalence $\hocolim_{\DE^{\op}} X_* \rtarr |X_*|$ for each $X_* \in s\sT$.
Under this lens, Theorem \ref{thm:rezk} is manifestly a question of the preservation of homotopy pullbacks under certain kinds of homotopy colimits. For the rest of this section, we will use $|X_*|$ and $\hocolim_{\DE^{\op}}X_*$ interchangeably (and by one, we may as well mean the other). By a \emph{realization-fibration} (sometimes denoted RF), we will mean a map $f_* \colon Y_* \rtarr B_*$ such that
the output of Theorem \ref{thm:rezk} holds true (i.e. a realization-fibration is a map such that homotopy colimits commute with homotopy pullbacks of the map). Our task then is to ascertain conditions under which a map of simplicial spaces is a realization-fibration.\\

The perspective due to Rezk allows us to deduce Theorem \ref{thm:rezk} from a local-to-global principle for realization-fibrations, applied to the homotopy colimit of representable sifted presheaves. Our starting point for the local-to-global point of view is the principle of descent, as presented in \cite[Section 6.5 and Proposition 6.6]{Rezk2010} for model toposes. 
We say that a natural transformation $f \colon U \rtarr V$ of functors $U,V \colon \sJ \rtarr \sT$ is equifibered if for every $J \rtarr J'$ in $\sJ$, the associated naturality square is a homotopy pullback square. We will use the notation $(\aC,\aD)$ to mean the functor category whose objects are functors $\aC \rtarr \aD$ and morphisms are natural transformations.
\begin{lem}[Descent]\label{lem:descent}
    \begin{enumerate}
        Let $\sJ$ be a (small) category.
        \item Let $p \colon E \rtarr B$ be a fibration in $\sT$ and $V \colon \sJ \rtarr \sT$ be a diagram of spaces such that $h \colon \hocolim_{\sJ} V \rtarr B$ is a weak equivalence in $\sT$. Let $U \colon \sJ \rtarr \sT$ be the diagram of spaces defined by $U(J) = V(J) \times_B E$. The map $\hocolim_{\sJ} U \rtarr E$ is a weak equivalence.
        \item Let $U$ and $V$ be diagrams of spaces $\sJ \rtarr \sT$, and $f \colon U \rtarr V$ an equifibered map in $(\sJ, \sT)$. For each $J \in \sJ$, the square:
        \[
        \xymatrix{
        U(J) \ar[d] \ar[r] & \hocolim_{\sJ} U \ar[d]\\
        V(J) \ar[r] & \hocolim_{\sJ} V
        }
        \]
        \noindent is a homotopy pullback.
        \item Let $f \colon X \rtarr Y$ be an equifibered map, and \[\xymatrix{
            X' \ar[r] \ar[d] & X \ar[d]^f\\
            Y' \ar[r] & Y
        }\]
        \noindent a homotopy pullback square in $(\sJ,\sT)$. Applying $\hocolim_{\sJ}$ at each corner of the above square still yields a homotopy pullback.
    \end{enumerate}
\end{lem}

\indent The central local-to-global statement we seek states that under certain hypotheses (i.e. equifiberedness) realization-fibrations can be glued together to get a realization-fibration.

\begin{lem}\label{lem:localtoglobal}
Let $\sJ$ be a small category, $h \colon W \rtarr V$ an equifibered map in $(\sJ, s\sT)$ that is an object-wise realization-fibration. Then the map obtained by passing to the homotopy colimit $\hocolim_{\sJ} h \colon \hocolim_{\sJ} W \rtarr \hocolim_{\sJ} V$ is also a realization-fibration.
\end{lem}
\begin{proof}
The argument in this proof is due to \cite{Rezk14}. Define $B_* := \hocolim_{\sJ} V$ and pick a factorization of $\hocolim_{\sJ} h \colon \hocolim_{\sJ} W \rtarr B_*$ as the composite $\xymatrix@1{\hocolim_{\sJ} W \ar^-{\io_*}[r] & E_* \ar[r]^{p_*} & B_*}$, where $\io_*$ is a weak equivalence and $p_*$ is a fibration.
Note that to show $\hocolim_{\sJ} h$ is an RF, it suffices to show $p_*$ is an RF. Toward that end, we make some initial remarks.
For each $J \in \sJ$, we have a map $V(J)_* \rtarr B_*$. Denote by $U(J)_*$ the fiber product $E_* \times_{B_*} V(J)_*$.
\[
\xymatrix{
    U(J)_* \ar[r] \ar[d] & E_* \ar[d]^-p\\
    V(J)_* \ar[r] & B_*
}
\] 
\noindent 
$h \colon W \rtarr V$ is equifibered, so Lemma \ref{lem:descent} tells us that there is a homotopy pullback square.
\[
\xymatrix{
    W(J)_* \ar[r] \ar[d] & \hocolim_{\sJ} W \ar[d]^-{\hocolim_{\sJ}h}\\
    V(J)_* \ar[r] & B
}
\] 
Since $\io_* \colon \hocolim_{\sJ} W \rtarr E_*$ is a weak equivalence, we see that there is a weak equivalence $W(J)_* \rtarr U(J)_*$ for each $J \in \sJ$. In particular, the maps $U(J)_* \rtarr V(J)_*$ are RF too.
Lemma \ref{lem:descent} also provides a chain of weak equivalences $\hocolim_{\sJ} W \rtarr \hocolim_{\sJ} U \rtarr E_*$. We are ready to show that $p_*$ is an RF now.\\

Let $\ph \colon B'_* \rtarr B_*$ be a map in $s\sT$. Define $E'_*$, $U'$, and $V'$ by the following homotopy pullbacks.
\[\xymatrix{
    E'_* \ar[r] \ar[d] & E_* \ar[d]^{p}\\
    B'_* \ar[r]_{\ph} & B_*
} \qquad
\xymatrix{
    V'(J)_* \ar[r] \ar[d] & V(J)_* \ar[d]\\
    B'_* \ar[r]_{\ph} & B_*
} \qquad
\xymatrix{
    U'(J)_* \ar[r] \ar[d] & U(J)_* \ar[d]\\
    B'_* \ar[r]_{\ph} & B_*
}
\]
\noindent By observing the middle and the right homotopy pullbacks, and using the homotopy pullback lemma, we get the following homotopy pullback square. 
\[
\xymatrix{
    U' \ar[r] \ar[d] & U \ar[d]\\
    V' \ar[r] & V
}
\]

We already know that there is a weak equivalence $\hocolim_{\sJ} U \rtarr E_*$, and an equality $B_* = \hocolim_{\sJ} V$. But by Lemma \ref{lem:descent}, there are also weak equivalences $\hocolim_{\sJ}V' \rtarr B'_*$ and $\hocolim_{\sJ} U' \rtarr E'_*$. This tells us that applying $\hocolim_{\sJ}$ to each corner of the above square, 
we get a homotopy pullback 
\[\xymatrix{
    E'_* \ar[r] \ar[d] & E_* \ar[d]^{p}\\
    B'_* \ar[r]_{\ph} & B_*
}\]
\noindent It suffices to know that applying $|-|$ at each corner, we still get a homotopy pullback square. But notice that $|\hocolim_{\sJ} (-)| \simeq \hocolim_{\sJ} |-|$ since colimits commute with colimits. So we might as well have realized down in our earlier homotopy pullback, and then applied $\hocolim_{\sJ}$. Note that since each $U(J)_* \rtarr V(J)_*$ is an RF, upon realization we get a homotopy pullback.
\[
\xymatrix{
    |U'| \ar[r] \ar[d] & |U| \ar[d]\\
    |V'| \ar[r]& |V|
}
\]
\noindent Now we notice that $|U| \rtarr |V|$ is equifibered since each $U(J)_* \rtarr V(J)_*$ is an RF. Therefore by descent, applying $\hocolim_{\sJ}$ to each corner will yield yet another homotopy pullback square, which is what we needed. 
\end{proof}
\indent We now make use of the siftedness of $\DE^{\op}$. When we say sifted, we implicitly mean the kind of homotopy siftedness (after a routine check in homotopy coend calculus using the Bousfield-Kan formula) that results in a canonical weak equivalence $|X_* \times Y_*| \rtarr |X_*| \times |Y_*|$ for any $X_*, Y_*$ in $s\sT$.
In simple terms, siftedness tells us that realization commutes with finite products in $s\sT$. The primary utility of working over a sifted category is demonstrated in the following lemma.
\begin{lem}\label{lem:siftedprojection}
All maps of the form $p_* \colon B_* \times C_* \rtarr B_*$ in $s\sT$ are realization-fibrations.
\end{lem}
\begin{proof}
Let $\ph_* \colon B_*' \rtarr B_*$ be a map in $s\sT$. Any homotopy pullback of $p_*$ along $\ph_*$ must look like the following diagram.
\[
\xymatrix{
    B'_* \times C_* \ar[r]^{\ph_* \times \id} \ar[d]_{\pi_{B'_*}} & B_* \times C_* \ar[d]^{p_*}\\
    B'_* \ar[r]_{\ph_*} & B_*
}
\]
\noindent where $\pi_{B'_*}$ is the projection map onto $B'_*$. Applying $|-|$ on each corner, and using the canonical weak equivalences $|B_* \times C_*| \rtarr |B_*| \times |C_*|$ and $|B'_* \times C_*| \rtarr |B'_*| \times |C_*|$, we get a diagram.
\[
\xymatrix{
    |B'_*| \times |C_*| \ar[rr]^{|\ph_*|\times\id} \ar[d]_{\pi_{|B'_*|}} & & |B_*| \times |C_*| \ar[d]^{|p_*|}\\
    |B'_*| \ar[rr]_{|\ph_*|} && |B_*|
}
\]
\noindent where $\pi_{|B'_*|}$ is the projection map onto $|B'_*|$. This is evidently another homotopy pullback square, and we are done.
\end{proof}

Now we introduce the notions of weak and local projection maps. We will soon see that every local projection map is in fact a realization-fibration, and that the hypotheses in Theorem \ref{thm:rezk} ensure that $f_* \colon Y_* \rtarr B_*$ will be a local projection.
A map $p_* \colon E_* \rtarr B_*$ in $s\sT$ is called a \emph{weak projection} if it is weakly equivalent over $B_*$ to a projection map $\tilde{p}_* \colon B_* \times C_* \rtarr B_*$ in $s\sT$. We will say that a map $p_* \colon E_* \rtarr B_*$ in $s\sT$ is a \emph{local projection} if for every 
$n$, and every homotopy pullback square of the form:
\[
\xymatrix{
    E'_* \ar[r] \ar[d]_-{q_*} & E_* \ar[d]^-{p_*}\\
    \DE[n] \ar[r]_-{b} & B_*
}
\]
\noindent where $\DE[n]$ denotes the standard simplicial set $\bm \mapsto \Hom_{\DE}(\bm,\bn)$ (thought of as a simplicial space by the constant scheme construction), the map $q_*$ is a weak projection map. Given $X_* \in s\sT$, there is a simplicial presheaf $\pi_0^{\bA^1}(X_*)$ given by $(\pi_0^{\bA^1} (X_*))_n = \pi_0^{\bA^1} (X_n)$. We can think of $\pi_0^{\bA^1}(X_*)$ as a levelwise discrete object in $s\sT$. Given a map $p_* \colon E_* \rtarr B_*$, 
let $\lproj(p)_n \subseteq (\pi_0^{\bA^1}(B_*))_n$ denote the presheaf prescribed locally as the collection of sections of $(\pi_0^{\bA^1}(B_*))_n$ represented by maps $b_* \colon \DE[n] \rtarr B_*$ such that the pullback of $p_*$ along $b_*$ is a weak projection. This prescription defines a simplicial subpresheaf $\lproj(p_*) \subseteq \pi_0^{\bA^1}(B_*)$.\\

Suppose $p_* \colon E_* \rtarr B_*$ and $\ph_* \colon B'_* \rtarr B_*$ are maps in $s\sT$. Note that by the pullback lemma, the pullback of $p_*$ along $\ph_*$ is a local projection iff $\ph(\pi_0^{\bA^1} B'_*) \subseteq \lproj(p_*)$. Taking $B'_*=B_*$ and $\ph_* = \id_{B_*}$, this means in particular that $p_*$ is a local projection iff $\lproj(p_*) = \pi_0^{\bA^1} (B_*)$. In applications, we typically demonstrate that $p_*$ is a local projection by showing that 
$\pi_0^{\bA^1} (B_*) \subseteq \lproj(p_*)$. 

\begin{lem}
    All local projections in $s\sT$ are realization-fibrations.
\end{lem}
\begin{proof}
    Let $p_* \colon E_* \rtarr B_*$ be a local projection in $s\sT$. Note that we can write $B_*$ up to weak equivalence as the colimit of representable objects in $s\sT$ (by which we mean $\DE[n]$ for $n \geq 0$, considered as simplicial spaces). 
    For each map $\DE[n] \rtarr B_*$ forming the appropriate colimit, since $p_*$ is a local projection, we can form the homotopy pullback square below.
    \[
    \xymatrix{
        \DE[n] \times C_* \ar[r] \ar[d] & E_* \ar[d]^p\\
        \DE[n] \ar[r] & B_*
    }
    \]
    \noindent Descent tells us that $E_*$ then is the homotopy colimit of objects of the form $\DE[n] \times C_*$. In sum, $p_*$ can be written as the homotopy colimit of maps that each look like $\DE[n] \times C_* \rtarr \DE[n]$. By Lemma \ref{lem:siftedprojection}, we know that each of these maps is
    a realization-fibration. Lemma \ref{lem:localtoglobal} then tells us that $p_*$ is a realization-fibration too. 
\end{proof}

\begin{proof}[Proof of Theorem 3.10]
    We return to the notation set up in the statement of Theorem \ref{thm:rezk}. Namely, let $f_* \colon Y_* \rtarr B_*$ be a map in $s\sT$.
    Note first that since $\DE[0]$, as a simplicial space, is terminal, every map into $\DE[0]$ in $s\sT$ is a weak projection. As a result, any 0-simplex of $\pi_0^{\bA^1} (B_*)$, thought of as being represented by a map out of $\DE[0]$, is contained in $\lproj(f_*)$.
    Since $\pi_0^{\bA^1} (B_*)$ is discrete by assumption, meaning $\pi_0^{\bA^1}(B_*) \simeq \mathrm{const}(\pi_0^{\bA^1}(B_0))$, we see in fact that $\pi_0^{\bA^1}(B_*) \subseteq \lproj(f_*)$. We conclude that $f_*$ is a local projection, and therefore a realization-fibration.
\end{proof}

\indent We complete the proof of Lemma \ref{lem:3.1} by introducing, and passing to, fibrant replacements. Recall that $\sT$ admits functorial factorization. It follows from a result on Reedy model structures that this property lifts to simplicial spaces as well, so that $s\sT$ admits functorial factorization too (see \cite[Theorem 5.2.5]{Hovey99} for the general statement due to transfinite induction). The existence of a fibrant replacement functor $\Ex \colon s\sT \rtarr s\sT$ with the following properties is now immediate.
\begin{enumerate}
\item For every $X_* \in s\sT$, $\Ex(X_*)$ is fibrant in $s\sT$.
\item For every $X_* \in s\sT$, there is an acyclic cofibration $\mu_* \colon X_* \rtarr \Ex(X_*)$.
\item $\Ex$ preserves (and in fact reflects) weak equivalences\footnote{This follows from applications of the two-out-of-three property for weak equivalences.}. 
\end{enumerate}

\noindent Let $X_*$ be an arbitrary simplicial space that is Reedy cofibrant and level-wise connected. Consider the following naturality square.
\[
\xymatrix{
|\OM_* X_*| \ar[r]^{\ga} \ar[d]_-{|\OM_* \mu_*|} & \OM |X_*| \ar[d]^-{\OM |\mu_*|}\\
|\OM_* \Ex X_*| \ar[r]_{\ga} & \OM |\Ex X_*|
}
\]
\noindent The bottom arrow of this square (by Lemma \ref{lem:3.1.2}) is a weak equivalence since $\Ex X_*$ is fibrant. The left and right arrows of the square are weak equivalences since $\OM$ preserves weak equivalences and $|-|$ preserves weak equivalences between Reedy cofibrant objects. By the two-out-of-three property for weak equivalences, it follows that the top arrow $\ga \colon |\OM_* X_*| \rtarr \OM |X_*|$ is a weak equivalence.

\printbibliography
\end{document}